\documentclass{amsproc}
\usepackage{amssymb}
\usepackage{amsfonts}
\usepackage{xcolor}
\usepackage{cite}
\usepackage{graphicx}

\theoremstyle{plain}

\newtheorem{claim}{Claim}

\newtheorem{notation}{Notation}

\newtheorem{theorem}{Theorem}[section]
\newtheorem{definition}[theorem]{Definition}
\newtheorem{proposition}[theorem]{Proposition}
\newtheorem{corollary}[theorem]{Corollary}
\newtheorem{lemma}[theorem]{Lemma}
\newtheorem{example}[theorem]{Example}
\newtheorem{remark}[theorem]{Remark}

\numberwithin{equation}{section}
\newtheorem*{theorem*}{Theorem}

\input{tcilatex}

\begin{document}
\title[]{Intermediate Assouad-like Dimensions}
\author{Ignacio Garc\'{\i}a, Kathryn E. Hare and Franklin Mendivil}
\address{Centro Marplatense de Investigaciones Matem\'{a}ticas, Facultad de
Ciencias Exactas y Naturales\\
and Instituto de Investigaciones F\'{\i}sicas de Mar del Plata (CONICET)\\
Universidad Nacional de Mar del Plata, Argentina }
\email{nacholma@gmail.com}
\address{Dept. of Pure Mathematics, University of Waterloo, Waterloo, Ont.,
Canada, N2L 3G1}
\email{kehare@uwaterloo.ca}
\address{Department of Mathematics and Statistics, Acadia University,
Wolfville, N.S. Canada, B4P 2R6}
\email{franklin.mendivil@acadiau.ca}
\thanks{The research of K. Hare is partially supported by NSERC 2016:03719.
The research of F. Mendivil is partially supported by NSERC\ 2012:238549. I.
Garc\'{\i}a and K. Hare thank Acadia University for their hospitality when
some of this research was done.}

\maketitle

\begin{abstract}
We study a class of bi-Lipschitz-invariant dimensions that range between the
box and Assouad dimensions. The quasi-Assouad dimensions and $\theta $%
-Assouad spectrum are other special examples. These dimensions are
localized, like Assouad dimensions, but vary in the depth of scale which is
considered, thus they provide very refined geometric information. Our main
focus is on the intermediate dimensions which range between the
quasi-Assouad and Assouad dimensions, complementing the $\theta $-Assouad
spectrum which ranges between the box and quasi-Assouad dimensions.

We investigate the relationship between these and the familiar dimensions.
We construct a Cantor set with a non-trivial interval of dimensions, the
endpoints of this interval being given by the quasi-Assouad and Assouad
dimensions of the set. We study stability and continuity-like properties of
the dimensions. In contrast with the Assouad-type dimensions, we see that
decreasing sets in $\mathbb{R}$ with decreasing gaps need not have dimension 
$0$ or $1$. As is the case for Hausdorff and Assouad dimensions, the Cantor
set and the decreasing set have the extreme dimensions among all compact
sets in $\mathbb{R}$ whose complementary set consists of open intervals of
the same lengths.
\end{abstract}

\section{Introduction and main results}

\subsection{Introduction}

Over the years, many notions of dimension have been introduced to help
understand the geometry of (often `small') subsets of metric spaces, such as
subsets of $\mathbb{R}^{n}$ of Lebesgue measure zero. Hausdorff, box and
packing dimensions are well known examples of such notions. More recently,
the upper and lower Assouad dimensions of a set $E$, denoted $\dim _{A}E$
and $\dim _{L}E$ respectively, which quantify the `thickest' or `thinnest'
part of the space, were introduced by Assouad in {\cite{A1,A2} and Larman in 
\cite{L1}. Along with their less extreme versions, the upper and lower
quasi-Assouad dimensions, }$\dim _{qA}E$ and $\dim _{qL}E,$ introduced in 
\cite{CDW, LX},{\ these dimensions have been extensively studied within the
fractal geometry community; see for example, \cite{CWC, FTrans, FHOR, FTale,
FYAdv, FYInd, GH, HT18, KR, Li, Luu, MT} and the references cited therein.
These dimensions can roughly be thought of as local refinements of the
box-counting dimensions where one takes the most extreme local behaviour.
The following relationships are known for all compact sets }$E$:%
\begin{equation*}
\dim _{L}E\leq \dim _{qL}E\leq \dim _{H}E\leq \underline{\dim }_{B}E\leq 
\overline{\dim }_{B}E\leq \dim _{qA}E\leq \dim _{A}E,
\end{equation*}%
where $\dim _{H}$, $\underline{\dim }_{B}$, $\overline{\dim }_{B}$ denote
the Hausdorff, lower and upper box dimensions respectively. See \cite{Fal,
FTrans, GH,LX} for proofs.

In several recent papers, a range of intermediate dimensions have been
introduced and studied. For instance, in \cite{FFK}, Falconer, Fraser and
Kempton discussed a continuum of dimensions that lie between the Hausdorff
and box dimensions. In \cite{FYAdv, FYInd}, Fraser and Yu focussed on the
family of dimensions known as the upper (or lower) $\theta $-Assouad
spectrum, which lie between the upper (or lower) box and quasi-upper (resp.,
lower) Assouad dimensions. For further references, we refer the reader to
Fraser's survey paper \cite{FrSurvey}.

In this paper, we study a general class of intermediate dimensions, which we
refer to as the upper and lower $\Phi $-dimensions. These include the
(quasi-) Assouad dimensions and $\theta $-Assouad spectrum as special cases,
with the box dimensions typically arising as a limit. More generally, the $%
\Phi $-dimensions provide a range of bi-Lipschitz invariant dimensions
between the box, quasi-Assouad and Assouad dimensions. As the box and
(quasi-) Assouad dimensions for a given set can all be different, the
intermediate $\Phi $-dimensions provide more refined information about the
local geometry of the set, such as detailed information about the scales at
which one can observe extreme local behaviour.

Another motivation for us to investigate these intermediate dimensions was
the classical problem of understanding the dimension of `rearrangements' of
Cantor sets. This problem was first considered by Besicovitch and Taylor in 
\cite{BT} for the Hausdorff dimension in the deterministic case and later by
Hawkes in \cite{Ha} for the random situation. In \cite{GHMArXiv} we prove
that if $\Phi (x)\ll \log |\log x|/|\log x|$, then the upper and lower $\Phi 
$-dimensions of almost all (in a natural probabilistic sense) rearrangements
of a given Cantor set are $1$ and $0$ respectively, while if $\Phi (x)\gg
\log |\log x|/|\log x|,$ then almost all rearrangements have the same upper and lower $\Phi $%
-dimensions as the original Cantor set. The first case includes the Assouad
dimensions and the second, the quasi-Assouad. In \cite{Tr}, Troscheit
obtained similar results for other random constructions.

\subsection{$\Phi $-dimensions}

To explain these dimensions in more detail, we first recall that f{or the
upper box dimension of a metric space }$E$ {one considers the minimal number
of balls of radius }$r$ that are required to cover the entire space $E,$ say 
$N_{r}(E),$ and computes the infimal exponent $s$ such that $N_{r}(E)\leq
r^{-s}$ as $r\rightarrow 0$. For the upper Assouad dimension of $E$ one
determines, instead, the infimal $s$ such that 
\begin{equation*}
N_{r}(B(z,R)\cap E)\leq (R/r)^{s}
\end{equation*}%
for all $r\leq R$ and all centres $z\in E$. The lower Assouad dimension is a
similar local variation of the lower box dimension. The quasi-Assouad
dimensions are less extreme versions of the Assouad dimensions, requiring
only that the bounds hold for $r\leq R^{p}$ where the exponent $p$ decreases
to $1$.

Fraser and Yu observed in \cite[Section 9]{FYAdv} that a rich dimension
theory can be developed by considering decreasing continuous functions $%
F(x)\leq x$, choosing $r=F(R)$ (or $r\leq F(R)$ in our modified case) and
studying the corresponding Assouad-like dimensions. In \cite{FTale, FYAdv,
FYInd}, Fraser et al studied the special case of $F(x)=x^{1/\theta }$ for
fixed $\theta \in (0,1)$, the so-called upper and lower $\theta $-Assouad
spectrum. As $\theta \rightarrow 1,$ the upper (lower) $\theta $-Assouad
spectrum tends from below (resp., above) to the upper (resp., lower)
quasi-Assouad dimension. As $\theta \rightarrow 0,$ the upper $\theta $%
-Assouad spectrum tends from above to the upper box dimension, while the
lower $\theta $-Assouad spectrum is dominated by the lower box dimension.

Motivated by the work of Fraser, we consider the quite general class of
functions $F(x)=x^{1+\Phi (x)},$ requiring only that $F(x)$ decreases to $0$
as $x\downarrow 0.$ The \textbf{upper }and\textbf{\ lower }$\Phi $\textbf{%
-dimensions }of $E$, denoted by $\overline{\dim }_{\Phi }E,$ $\underline{%
\dim }_{\Phi }E$ respectively, arise by restricting to $r\leq R^{1+\Phi (R)}$%
. We refer the reader to Definition \ref{DefnPhidim} for the precise
definitions.

When $\Phi =0$ we recover the Assouad dimensions and when $\Phi =1/\theta
-1, $ we get the $\theta $-Assouad spectrum. It will be shown that if $\Phi
(x)\rightarrow \infty $ as $x\rightarrow 0$ (and $\underline{\dim }_{\Phi
}E>0)$ the upper (lower) $\Phi $-dimension is the upper (resp., lower) box
dimension (Prop. \ref{Prop:box}), while if $\Phi (x)\rightarrow \delta \in
(0,\infty )$, then the upper (lower) $\Phi $-dimension coincides with the
upper (resp., lower) $\theta $-Assouad spectrum for $\theta =(1+\delta
)^{-1} $ (Cor. \ref{Coro:p}). Thus our main interest in this paper is in the
case that $\Phi \rightarrow 0$ (such as the function $\Phi (x)=\log |\log
x|/|\log x|,$ which appears naturally in the random problem) when we obtain
a full range of intermediate dimensions between the quasi-Assouad and
Assouad dimensions.

\subsection{Summary of the main results}

The primary purpose of this paper is to study the basic properties of these
intermediate dimensions. 

One easy property is that the upper $\Phi $%
-dimension is finitely stable, but the lower $\Phi $-dimension is not. See
Proposition \ref{union}.

\subsubsection{Relationship between dimensions}

A\ natural question to ask is how the $\Phi $-dimensions compare, both with
each other and to the familiar dimensions. Clearly, they are naturally
ordered: If $\Phi _{1}\leq \Phi _{2}$, then $\overline{\dim }_{\Phi
_{1}}(E)\geq \overline{\dim }_{\Phi _{2}}(E)$ and vice versa for the lower $%
\Phi $-dimensions. Thus the upper (lower) $\Phi $-dimensions lie between the
upper (lower) box and the upper (lower) Assouad dimensions.

Here is an overview of some of our main theoretical results on this question.

\begin{enumerate}
\item If $\Phi (x)\leq c/\left\vert \log x\right\vert ,$ then the upper and
lower $\Phi $-dimensions coincide with the upper and lower Assouad
dimensions (respectively) for all sets $E$. See Proposition \ref{ADim}.

\smallskip

\item The upper and lower quasi-Assouad dimensions are special examples of
upper and lower $\Phi $-dimensions, but the choice of dimension functions
will depend on the underlying set $E$. See Proposition \ref{qA}.

\smallskip

\item If $\Phi _{1}/\Phi _{2}(x)\rightarrow 1$ as $x\rightarrow 0,$ then the
upper (and lower) $\Phi _{1}$ and $\Phi _{2}$-dimensions coincide for all
sets $E$. See Proposition \ref{Propo-equal}(i).

\smallskip

\item \label{S4}If there is a constant $\xi >0$ such that $\Phi _{1}\geq
(1+\xi )\Phi _{2},$ then there are sets $E_{1},E_{2}$ with $\overline{\dim }%
_{\Phi _{1}}(E_{1})\neq $ $\overline{\dim }_{\Phi _{2}}E_{1}$ and $%
\underline{\dim }_{\Phi _{1}}E_{2}\neq \underline{\dim }_{\Phi _{2}}E_{2}$.
See Theorem \ref{Diff}.

\smallskip

\item \label{S5}Given any family of decreasing dimension functions, $\{\Phi
_{p}\}_{p\in (0,1)},$ with $\Phi _{p}\gg \Phi _{q}$ for $p>q,$ and given any
decreasing, continuous function $d:(0,1)\rightarrow $ $[a,b]\subseteq (0,1),$
there is a set $E\subseteq \mathbb{R}$ with $\overline{\dim }_{\Phi
_{p}}(E)= $ $d(p)$ for all $p$. The analogous result holds for the lower
dimensions. See Theorem \ref{continuum}. Hence there are subsets of $%
\mathbb{R}$ with a full (non-trivial) interval of dimensions whose endpoints
are given by the quasi-Assouad and Assouad dimensions. See Corollary \ref%
{Cor:continuum}.

\smallskip

\item It is known that the lower $\theta $-Assouad spectrum of a set are not
uniformly dominated above by the Hausdorff dimension. However $\underline{%
\dim }_{\Phi _{\theta }}E\leq \frac{1}{\theta }\dim _{H}E$ for $\Phi
_{\theta }=1/\theta -1$. See Proposition \ref{Propo-qLHau}.

\smallskip

\item In \cite{CWC} and \cite{FYInd}, it is shown that both the upper
and lower $\theta $-Assouad spectrum are continuous in the parameter $\theta 
$. More generally, if $\Phi _{t}(x)=g(t)\Phi (x)$ where $g$ is continuous
and $g(t_{0})\neq 0$, then $\overline{\dim }_{\Phi _{t}}(E)\rightarrow 
\overline{\dim }_{\Phi _{t_{0}}}(E)$ as $t\rightarrow t_{0},$ and similarly
for the lower dimensions. But this need not be true when $g(t_{0})=0$. See
Propositions \ref{Propo-equal}(ii) and \ref{ContFailure}. This suggests
that it may be difficult to find a one-parameter family of continuous
dimension functions that interpolates precisely between the quasi-Assouad
and Assouad dimensions.
\end{enumerate}

\subsubsection{Decreasing sets}

In \cite{GHM} it was shown that if $E=\{x_{n}\}_{n}\subseteq \mathbb{R}$ is
a decreasing sequence with decreasing gaps, then the Assouad dimension is
either $0$ or $1$. Likewise, the quasi-Assouad dimension is $0$ if $%
\overline{\dim }_{B}E=0$ and $1$ otherwise. In contrast, in Example \ref%
{Ex:Dec} we construct a decreasing set $E$ with decreasing gaps and a
dimension function $\Phi \rightarrow 0$ with $\dim _{qA}E=0<\overline{\dim }%
_{\Phi }E<\dim _{A}E=1$.

\smallskip

\subsubsection{Cantor sets and Rearrangements\label{rearr}}

We give formulas for the $\Phi $-dimensions of Cantor-like sets, similar to
those known for Hausdorff, box and Assouad dimensions, in Theorem \ref%
{Cantorformula}. These are used in some of our constructions, such as in
exhibiting sets with different values for various $\Phi $-dimensions as
mentioned in (\ref{S4}) and (\ref{S5}) above. This approach was taken in 
\cite{GH} and \cite{LX} to construct sets with different box and
(quasi-) Assouad dimensions, but new ideas are required here.

In \cite{BT}, Besicovitch and Taylor proved that if $C$ is a Cantor-like
set, then the interval $[0,\dim _{H}C]$ is precisely the set of Hausdorff
dimensions of `rearrangements' of $C,$ the sets whose complement consists of
open intervals of the same lengths as the complementary intervals of $C$.
This was extended in \cite{GHM} where it was found that the set of
attainable lower (upper) Assouad dimensions was \thinspace $\lbrack 0,\dim
_{L}C]$ (resp., $[\dim _{A}C,1])$, with $0$ (resp., $1)$ being the lower
(upper) Assouad dimension of the decreasing rearrangement. In Section \ref%
{Rearrangements}, we study this problem for the $\Phi $-dimensions and show
that under natural assumptions, the $\Phi $-dimensions of any rearrangement
lies between the dimensions of the Cantor set and the decreasing
rearrangement (whose upper $\Phi $-dimension could be $<1)$. Further, for
any $\Phi \rightarrow p\in \lbrack 0,\infty ]$ (including the quasi-Assouad
dimensions), this full range of values can be attained. New construction
techniques are needed to do this. In \cite{GHMArXiv} some of these results
are used in studying the $\Phi $-dimensions of random rearrangements.

\subsection*{Acknowledgements} The authors are grateful to the referee for the valuable suggestions made to improve the clarity in the exposition of the paper. 

\section{Basic properties of the $\Phi $-dimensions}

\subsection{Definitions\label{sec:Notation}}

We begin with notation and definitions. Let $X$ be a metric space.

\begin{notation}
We denote the closed ball centred at $z\in X$ with radius $R$ by $B(z,R)$.
For a bounded set $E\subseteq X$, the notation $N_{r}(E)$ will mean the
least number of balls of radius $r$ that cover $E$.
\end{notation}

\begin{definition}
A map $\Phi :(0,1)\mathbb{\rightarrow R}^{+}$ is called a \textbf{dimension
function} if $x^{1+\Phi (x)}$ decreases as $x$ decreases to $0$. We write $%
\mathcal{D}$ for the set of all dimension functions.
\end{definition}

Of course, $R^{1+\Phi (R)}\leq R$, so $R^{1+\Phi (R)}\rightarrow 0$ as $%
R\rightarrow 0$ for any dimension function $\Phi $. Special examples of
dimension functions include the constant functions $\Phi (x)=\delta \geq 0$
and the function $\Phi (x)=1/|\log x|$.

\begin{definition}
\label{DefnPhidim} Let $X$ be a metric space and $\Phi \in \mathcal{D}$. The 
\textbf{upper }and \textbf{lower }$\Phi $\textbf{-dimensions\ }of $%
E\subseteq X$ are given by 
\begin{align*}
\overline{\dim }_{\Phi }E=\inf \Bigl\{\alpha :(\exists c_{1},c_{2}>0)&
(\forall 0<r\leq R^{1+\Phi (R)}\leq R<c_{1})\text{ } \\
& \sup_{z\in E}N_{r}(B(z,R)\cap E)\leq c_{2}\left( \frac{R}{r}\right)
^{\alpha }\Bigr\}
\end{align*}%
and 
\begin{align*}
\underline{\dim }_{\Phi }E=\sup \Bigl\{\alpha :(\exists c_{1},c_{2}>0)&
(\forall 0<r\leq R^{1+\Phi (R)}\leq R<c_{1})\text{ } \\
& \inf_{z\in E}N_{r}(B(z,R)\cap E)\geq c_{2}\left( \frac{R}{r}\right)
^{\alpha }\Bigr\}.
\end{align*}
\end{definition}

A standard argument based upon the relationship between balls in
bi-Lipschitz spaces, shows that upper and lower $\Phi $-dimensions are
preserved under bi-Lipschitz maps. We also note that, as with the box
dimensions, a set and its closure have the same upper or lower $\Phi $%
-dimensions.

The upper and lower (quasi)-Assouad dimensions and spectrum can be expressed
in terms of $\Phi $-dimensions.

\begin{remark}
(i) The \textbf{upper} \textbf{Assouad }and \textbf{lower Assouad dimensions 
}of $E$, denoted $\dim _{A}E$ and $\dim _{L}E$ respectively, are the special
cases of the upper and lower $\Phi $-dimensions with $\Phi (x)=0$ for all $x$%
.

\smallskip

(ii) The \textbf{upper }and \textbf{lower }$\theta $-\textbf{Assouad spectrum%
}, $\overline{\dim }_{A}^{\theta }E$ and $\underline{\dim }_{L}^{\theta }E,$
are the special cases of the upper and lower $\Phi _{\theta }$-dimensions
with $\Phi _{\theta }(x)=1/\theta -1>0$ for all $x$. To be precise, the
upper and lower $\theta $-Assouad spectrum introduced in \cite{FYAdv} only
required consideration of $r=R^{1/\theta }$. However, it was shown in \cite%
{FTale} that if we denote this dimension by $\overline{\dim }_{A}^{=\theta
}E,$ then $\overline{\dim }_{A}^{\theta }E=\sup_{\psi \leq \theta }\overline{%
\dim }_{A}^{=\psi }E,$ with the analogous statement proven in \cite{CWC}
(see also \cite{HT18}) for the lower $\theta $-Assouad spectrum.

\smallskip

(iii) The \textbf{upper} \textbf{quasi-Assouad} and \textbf{lower
quasi-Assouad dimensions\ }are given by%
\begin{equation*}
\dim _{qA}E=\lim_{\theta \rightarrow 1}\overline{\dim }_{\Phi _{\theta }}E%
\text{ and }\dim _{qL}E=\lim_{\theta \rightarrow 1}\underline{\dim }_{\Phi
_{\theta }}E,
\end{equation*}%
where, again, $\Phi _{\theta }(x)=1/\theta -1>0$. In Proposition \ref{qA} we
will prove that the quasi-Assouad dimensions can also be obtained as $\Phi $%
-dimensions, but the choice of $\Phi $ depends on the set $E$.
\end{remark}

We refer the reader to the references cited in the introduction of this
paper for further background information on these various Assouad-like
dimensions.

\smallskip

It is easy to see that $\underline{\dim }_{\Phi }E=0$ whenever $E$ has an
isolated point, thus we need not have $\underline{\dim }_{\Phi }E\leq 
\underline{\dim }_{\Phi }F$ if $E\subseteq F$. This monotonicity property
does, however, hold for the upper $\Phi $-dimension, as does finite
stability. We show this first, along with bounds on the lower $\Phi $%
-dimension of finite unions.

\begin{proposition}
\label{union}(i) If $E\subseteq F$, then $\overline{\dim }_{\Phi }E\leq $ $%
\overline{\dim }_{\Phi }F$. Indeed, for $E,F\subseteq X,$ 
\begin{equation*}
\overline{\dim }_{\Phi }\left( E\cup F\right) =\max \left( \overline{\dim }%
_{\Phi }E,\overline{\dim }_{\Phi }F\right) .
\end{equation*}%
(ii) For all $E,F\subseteq X,$%
\begin{equation*}
\min (\underline{\dim }_{\Phi }E,\underline{\dim }_{\Phi }F)\leq \underline{%
\dim }_{\Phi }(E\cup F)\leq \max \left( \underline{\dim }_{\Phi }E,\overline{%
\dim }_{\Phi }F\right) .
\end{equation*}
\end{proposition}

\begin{proof}
(i) The fact that $\overline{\dim }_{\Phi }\left( E\cup F\right) \geq \max
\left( \overline{\dim }_{\Phi }E,\overline{\dim }_{\Phi }F\right) $ follows
easily from the observation that if $z\in E\cup F,$ say $z\in E,$ then 
\begin{equation*}
N_{r}(B(z,R)\cap (E\cup F))\geq N_{r}(B(z,R)\cap E).
\end{equation*}

To see the reverse inequality, first note that 
\begin{equation*}
N_{r}(B(z,R)\cap (E\cup F))\leq N_{r}(B(z,R)\cap E)+N_{r}(B(z,R)\cap F).
\end{equation*}%
Fix $\varepsilon >0$ and assume $z\in E$ and $r\leq R^{1+\Phi (R)}$ for
small $R$. Then there is a constant $c$ such that $N_{r}(B(z,R)\cap E)\leq
c\left( \frac{R}{r}\right) ^{d_{1}+\varepsilon }$ where $d_{1}=\overline{%
\dim }_{\Phi }E$. If $B(z,R)\cap F$ is empty, then trivially 
\begin{equation*}
N_{r}\left( B(z,R)\cap (E\cup F)\right) \leq c\left( \frac{R}{r}\right)
^{d_{1}+\varepsilon }.
\end{equation*}%
Otherwise, there is some $y\in B(z,R)\cap F,$ and as $B(z,R)\cap F\subseteq
B(y,2R)\cap F$ we have 
\begin{equation*}
N_{r}(B(z,R)\cap F)\leq N_{r}(B(y,2R)\cap F)\leq c^{\prime }\left( \frac{R}{r%
}\right) ^{d_{2}+\varepsilon }
\end{equation*}%
for $d_{2}=\overline{\dim }_{\Phi }F$. In either case, there is a constant $%
C $ such that 
\begin{equation*}
N_{r}(B(z,R)\cap (E\cup F))\leq C\left( \frac{R}{r}\right) ^{d+\varepsilon }
\end{equation*}%
for $d=\max (\overline{\dim }_{\Phi }E,\overline{\dim }_{\Phi }F)$ for all $%
z\in E\cup F$ and $r\leq R^{1+\Phi (R)}$, proving that $\overline{\dim }%
_{\Phi }\left( E\cup F\right) \leq d$.

(ii) The lower bound is obvious. For the upper bound, choose $z_{j}$ and $%
r_{j}\leq R_{j}^{1+\Phi (R_{j})}$ such that $N_{r_{j}}(B(z_{j},R_{j})\cap
E)\leq c_{1}\left( \frac{R_{j}}{r_{j}}\right) ^{d_{1}+\varepsilon }$ where $%
d_{1}=\underline{\dim }_{\Phi }E$ and $\varepsilon >0$ is fixed. Choosing $%
y_{j}\in B(z_{j},R_{j})\cap F$ (if this set is non-empty), we have 
\begin{eqnarray*}
N_{r_{j}}(B(z_{j},R_{j})\cap (E\cup F)) &\leq &N_{r_{j}}(B(z_{j},R_{j})\cap
E)+N_{r_{j}}(B(y_{j},2R_{j})\cap F) \\
&\leq &c_{1}\left( \frac{R_{j}}{r_{j}}\right) ^{d_{1}+\varepsilon
}+c_{2}\left( \frac{R_{j}}{r_{j}}\right) ^{d_{2}+\varepsilon }\leq C\left( 
\frac{R_{j}}{r_{j}}\right) ^{d+\varepsilon }
\end{eqnarray*}%
where $d_{2}=\overline{\dim }_{\Phi }F$ and $d=\max (d_{1},d_{2})$.
\end{proof}

\begin{remark}
We will always assume the underlying metric space $X$ is \textbf{doubling},
which means that there is a constant $M\geq 1$ so that for any $R>0$, each
ball of radius $R$ can be covered with at most $M$ balls of radius $R/2$.
The least such $M$ is called the \textbf{doubling constant} of the space.
This condition is equivalent to saying the space has bounded upper Assouad
dimension, \cite{Heinonen}. For example, if $E\subseteq \mathbb{R}^{d},$
then $\dim _{A}E\leq d$.

The doubling assumption ensures that in the definitions of the $\Phi $%
-dimensions the covering number $N_{r}(B(z,R)\cap E)$ could be replaced by
the packing number $P_{r}(B(z,R)\cap E)$, the maximum number of disjoint
balls of radius $r$, centred in $B(z,R)\cap E$. This is because the packing
and covering numbers are comparable: 
\begin{equation*}
\frac{1}{M}N_{r}(F)\leq P_{r}(F)\leq M\cdot N_{r/2}(F)\text{ for any }%
F\subseteq X.
\end{equation*}
\end{remark}

\subsection{Relationships between dimensions}

We begin by recalling the relationships between Assouad-like dimensions and
various classical dimensions. We denote by $\dim _{H}$, $\underline{\dim }%
_{B}$ and $\overline{\dim }_{B}$ the Hausdorff, lower box and upper box
dimensions respectively. The box dimensions are defined for bounded sets and
satisfy 
\begin{equation*}
\dim _{H}E\leq \underline{\dim }_{B}E\leq \overline{\dim }_{B}E.
\end{equation*}%
We refer to \cite{Fal} for the definitions and basic properties of these
dimensions.

Clearly we have the following relationships:%
\begin{equation*}
\dim _{L}E\leq \dim _{qL}E\leq \underline{\dim }_{B}E\leq \overline{\dim }%
_{B}E\leq \dim _{qA}E\leq \dim _{A}E
\end{equation*}%
and 
\begin{equation*}
\dim _{L}E\leq \underline{\dim }_{\Phi }E\leq \overline{\dim }_{\Phi }E\leq
\dim _{A}E.
\end{equation*}%
Since $\underline{\dim }_{B}E\geq \dim _{H}E,$ we obviously have 
\begin{equation*}
\dim _{H}E\leq \dim _{qA}E
\end{equation*}%
and for closed sets $E$ it is also true that 
\begin{equation*}
\dim _{qL}E\leq \dim _{H}E;
\end{equation*}%
see \cite{GH} (and \cite{L2} for the earlier result, $\dim _{L}E\leq \dim
_{H}E$). This inequality is not true in general, for example, if $\mathbb{Q}$
denotes the set of rational numbers in $[0,1]$, then $\dim _{H}\mathbb{Q}=0,$
but $\dim _{qL}\mathbb{Q}=\dim _{qL}\mathbb{R=}1$.

Obviously, if $\Phi (x)\leq \Psi (x)$ for all $x\in (0,1)$, then 
\begin{equation*}
\underline{\dim }_{\Phi }E\leq \underline{\dim }_{\Psi }E\ \ \text{ and }\ \ 
\overline{\dim }_{\Psi }\leq \overline{\dim }_{\Phi }E.
\end{equation*}%
Consequently, if $\Phi (x)\rightarrow 0$ as $x\rightarrow 0$, then the $\Phi 
$-dimensions give a range of dimensions between the Assouad and
quasi-Assouad type dimensions: 
\begin{equation*}
\dim _{L}E\leq \underline{\dim }_{\Phi }E\leq \dim _{qL}E\leq \dim
_{qA}E\leq \overline{\dim }_{\Phi }E\leq \dim _{A}E\text{ .}
\end{equation*}

\begin{remark}
We remark that in Section \ref{Examples} of this paper many examples are
constructed which demonstrate strictness in these inequalities. These are
based on the formulas given in Theorem \ref{Cantorformula} for the $\Phi $%
-dimensions of Cantor sets. The reader can also refer to \cite[Ex.\thinspace
16]{GH} or \cite[Ex.\thinspace 1.18]{LX} for similar constructions
illustrating the strictness of the relationship between the quasi-Assouad
and Assouad dimensions.
\end{remark}

Here are some additional facts about the relationships between these
dimensions.

\begin{proposition}
\label{Prop:box}(i) For any $\Phi \in \mathcal{D},$ 
\begin{equation*}
\underline{\dim }_{\Phi }E\leq \underline{\dim }_{B}E\leq \overline{\dim }%
_{B}E\leq \overline{\dim }_{\Phi }E.
\end{equation*}%
(ii) If $\Phi (x)\rightarrow \infty $ as $x\rightarrow 0$, then $\overline{%
\dim }_{\Phi }E=\overline{\dim }_{B}E$. If, in addition, $\underline{\dim }%
_{\Phi }E>0,$ then $\underline{\dim }_{\Phi }E=\underline{\dim }_{B}E$.
\end{proposition}

\begin{remark}
Since $\underline{\dim }_{\Phi }E=0$ if $E$ has an isolated point, it need
not be true that $\underline{\dim }_{\Phi }E=\underline{\dim }_{B}E$ under
only the assumption that $\Phi (x)\rightarrow \infty $.
\end{remark}

\begin{proof}
(i) As our metric space is assumed to be doubling, we even have $\overline{%
\dim }_{A}E<\infty $. Thus, we may suppose $b=\overline{\dim }_{B}E$ and $%
b-\varepsilon =\overline{\dim }_{\Phi }E$ for some $\varepsilon >0$. Given
small $r$, choose $R$ such that $r=R^{\beta (R)}$ where $\beta (R)=\max
(1+\Phi (R),4)$. It is easy to see that%
\begin{equation*}
N_{r}(E)\leq \sup_{z\in E}N_{r}\left( B(z,R)\cap E\right) \cdot N_{R}(E).
\end{equation*}%
If $R$ is small enough, then for some constant $c,$ 
\begin{equation*}
N_{r}(E)\leq c\left( \frac{R}{r}\right) ^{b-\varepsilon
/2}R^{-(b+\varepsilon )}=R^{-\beta (R)(b-\varepsilon /2+3\varepsilon
/(2\beta (R))}\leq cr^{-(b-\varepsilon /8)}.
\end{equation*}%
This implies $\overline{\dim }_{B}E\leq b-\varepsilon /8$, which is a
contradiction.

The arguments are similar to show the lower box dimension is an upper bound
on the lower $\Phi $-dimensions, but using packing numbers instead of
covering numbers, both for the lower $\Phi $-dimensions and the lower box
dimension.

(ii) If $\Phi(x) \rightarrow \infty $ as $x\rightarrow 0,$ then $\Phi (x)>p$
for any $p,$ provided $x$ is sufficiently small. Thus the monotonicity of
the $\Phi $-dimensions implies that if $\theta =(1+p)^{-1}$ and $\Psi (x)=p$%
, then $\overline{\dim }_{B}E\leq \overline{\dim }_{\Phi }E\leq \overline{%
\dim }_{\Psi }E=\overline{\dim }_{A}^{\theta }E.$ The result for the upper
dimension then follows since the upper $\theta $-Assouad spectrum converges
to $\overline{\dim }_{B}E$ as $p\rightarrow \infty $ ($\theta \rightarrow 0$%
), as a consequence of \cite[Theorem 2.1]{FTale} and \cite[Proposition 3.1]%
{FYAdv}.

For the lower dimension case, suppose that $\underline{\dim }_{\Phi }E=d>0$
and let $0<\epsilon <d/2$ be given. Since $\Phi (x)\rightarrow \infty $, we
have 
\begin{equation*}
\frac{R^{d/\epsilon -1}}{r}\geq \frac{R^{d/\epsilon -1}}{R^{1+\Phi (R)}}%
\rightarrow \infty \text{ implies }\frac{\left( \frac{R}{r}\right)
^{d-\epsilon }}{\left( \frac{1}{r}\right) ^{d-2\epsilon }}\rightarrow \infty
\end{equation*}%
as $R\rightarrow 0$, provided $r\leq R^{1+\Phi (R)}$. We note that for $%
r\leq R^{1+\Phi (R)}\leq \text{diam}(E)$, 
\begin{equation*}
N_{r}(E)\geq N_{r}(B(z,R)\cap E)\geq C\left( \frac{R}{r}\right) ^{d-\epsilon
}\geq Cr^{2\epsilon -d}
\end{equation*}%
and thus $\underline{\dim }_{B}E\geq d-2\epsilon $ for all such $\epsilon $.
This implies that $\underline{\dim }_{B}E\geq d$. Since we know that $%
\underline{\dim }_{\Phi }E\leq \underline{\dim }_{B}E$, we obtain equality.
\end{proof}

It is natural to ask when two dimension functions give rise to the same
dimensions for all sets $E$.

\begin{proposition}
\label{ADim}If there is some constant $c$ such that $\Phi (x)\leq c/|\log x|$%
, then $\dim _{L}E=\underline{\dim }_{\Phi }E$ and $\dim _{A}E=\overline{%
\dim }_{\Phi }E$.
\end{proposition}

\begin{proof}
This is simply due to the fact that if $\Phi (x)\leq c/|\log x|$, then there
are positive constants $a,b$ such that $a<R^{\Phi (R)}<b$ for all $R\in
(0,1) $.
\end{proof}

\begin{proposition}
\label{Propo-equal} \label{DimEqual}(i) Suppose $\Phi _{1},\Phi _{2}\in 
\mathcal{D}$ and $\Phi _{1}(x)/\Phi _{2}(x)\rightarrow 1$ as $x\rightarrow 0$%
. Then $\overline{\dim }_{\Phi _{1}}E=\overline{\dim }_{\Phi _{2}}E$ and $%
\underline{\dim }_{\Phi _{1}}E=\underline{\dim }_{\Phi _{2}}E$ for all sets
\thinspace $E$.

(ii) \label{cont}Assume $g:(0,1)\rightarrow \mathbb{R}^{+}$ is continuous at 
$t_{0}$ and $g(t_{0})\neq 0$. Suppose $\Phi \in \mathcal{D}$ and put $\Phi
_{t}(x)=g(t)\Phi (x)$. For any set $E,$ $\overline{\dim }_{\Phi
_{t}}E\rightarrow \overline{\dim }_{\Phi _{t_{0}}}E$ as $t\rightarrow t_{0}$%
. The same statement holds for the lower dimensions.
\end{proposition}

From (i), we immediately deduce the following.

\begin{corollary}
\label{Coro:p} If $\Phi (x)\rightarrow \delta $ as $x\rightarrow 0$ for some 
$0<\delta <\infty $, then for $\theta =(1+\delta )^{-1}$, 
\begin{equation*}
\overline{\dim }_{\Phi }E=\overline{\dim }_{A}^{\theta }E\ \text{ and }\ 
\underline{\dim }_{\Phi }E=\underline{\dim }_{L}^{\theta }E.
\end{equation*}
\end{corollary}

The proof of the proposition is an easy consequence of the estimates in the
Lemma below.

\begin{lemma}
\label{lemma:cont} Let $\Phi $ and $\Psi $ be dimension functions and assume
that for some $\epsilon >0$ there exists $x_{0}>0$ such that $|\Phi (x)/\Psi
(x)-1|\leq \epsilon $ for all $0<x\leq x_{0}$.

(i) Choose $c_{1},c_{2}$, depending on $\epsilon $, such that 
\begin{equation*}
N_{r}(B(z,R)\cap E)\leq c_{2}\left( \frac{R}{r}\right) ^{(\overline{\dim }%
_{\Phi }E+\epsilon )}
\end{equation*}%
for all $z\in E$ and $r\leq R^{1+\Phi (R)}\leq $ $R\leq c_{1}$. Then, there
is a constant $C$ such that for any $z\in E$ and $0<r\leq R^{1+\Psi (R)}\leq
R\leq c_{1}$ we have 
\begin{equation}
N_{r}(B(z,R)\cap E)\leq C\left( \frac{R}{r}\right) ^{(\overline{\dim }_{\Phi
}E+\epsilon )(1+\epsilon )}  \label{upper}
\end{equation}

(ii) Analogously, chose $c_1, c_2$ such that 
\begin{equation*}
N_{r}(B(z,R)\cap E)\geq c_{2}\left( \frac{R}{r}\right) ^{(\overline{\dim }%
_{\Psi }E-\epsilon )}
\end{equation*}
for all $z\in E$ and $r\leq R^{1+\Psi (R)}\leq $ $R\leq c_{1}$. Then, there
is some $c>0$ such that for any $z\in E$ and $0<r\leq R^{1+\Phi (R)}\leq
R\leq c_{1}$, we have 
\begin{equation}
N_{r}(B(z,R)\cap E)\geq c\left( \frac{R}{r}\right) ^{\underline{\dim }_{\Psi
}E-\epsilon (1+\log _{2}M)},  \label{lower}
\end{equation}%
where $M$ is the doubling constant of the space.
\end{lemma}

\begin{proof}
(i) Let $d=\overline{\dim }_{\Phi }E$ and pick $0<r\leq R^{1+\Psi (R)}<R\leq
c_{1}$. If $r\leq R^{1+\Phi (R)}$, then (\ref{upper}) follows by the
definition of $\overline{\dim }_{\Phi }E$. Otherwise, $R^{1+\Phi (R)}<r\leq
R^{1+\Psi (R)}$, hence $\Psi (R)<\Phi (R)$ and therefore $0<\Phi (R)-\Psi
(R)\leq \epsilon \Psi (R)$. Consequently, 
\begin{align*}
N_{r}(B(z,R)\cap E)& \leq N_{R^{1+\Phi (R)}}(B(z,R)\cap E)\leq C\left( \frac{%
R}{R^{1+\Phi (R)}}\right) ^{d+\epsilon } \\
& \leq C\left( \frac{R}{r}\right) ^{d+\epsilon }R^{-(\Phi (R)-\Psi
(R))(d+\epsilon )} \\
& \leq C\left( \frac{R}{r}\right) ^{d+\epsilon }R^{-\epsilon \Psi
(R)(d+\epsilon )}\leq C\left( \frac{R}{r}\right) ^{(d+\epsilon )(1+\epsilon
)}.
\end{align*}

(ii) Now let $d=\underline{\dim }_{\Psi }E$ and pick $0<r\leq R^{1+\Phi
(R)}<R\leq c_{1}$. Again, if $r\leq R^{1+\Psi (R)}$, then (\ref{lower})
follows by the definition of $\underline{\dim }_{\Psi }E$. Otherwise, $%
R^{1+\Psi (R)}<r\leq R^{1+\Phi (R)}$, hence $\Phi (R)<\Psi (R)$ so $\Psi
(R)(1-\epsilon )\leq \Phi (R)$. Then, 
\begin{equation*}
N_{r}(B(z,R)\cap E)\geq N_{R^{1+\Phi (R)}}(B(z,R)\cap E)\geq
N_{R^{1+(1-\epsilon )\Psi (R)}}(B(z,R)\cap E).
\end{equation*}%
But $R^{1+(1-\epsilon )\Psi (R)}=R^{1+\Psi (R)}2^{T}$ with $T=\epsilon \Psi
(R)\log _{2}R^{-1}$, and since the metric space is doubling with doubling
constant $M$, iterating this definition we have that each ball of radius $%
R^{1+(1-\epsilon )\Psi (R)}$ can be covered by at most $M^{\lceil T\rceil }$
balls of radius $R^{1+\Psi (R)}$. Therefore, 
\begin{align*}
N_{R^{1+(1-\epsilon )\Psi (R)}}(B(z,R)\cap E)& \geq M^{-\lceil T\rceil
}N_{R^{1+\Psi (R)}}(B(z,R)\cap E) \\
& \geq cM^{-\lceil T\rceil }\left( \frac{R}{R^{1+\Psi (R)}}\right)
^{d-\epsilon }=cR^{\epsilon \Psi (R)\log _{2}M}R^{-\Psi (R)(d-\epsilon )} \\
& =cR^{-\Psi (R)(d-\epsilon (1+\log _{2}M))}\geq c\left( \frac{R}{r}\right)
^{d-\epsilon (1+\log _{2}M)}.
\end{align*}%
Here the last inequality holds since $R^{1+\Psi (R)}<r$.
\end{proof}

\begin{proof}[Proof of Proposition \protect\ref{Propo-equal}]
(i) The assumption $\Phi _{1}(x)/\Phi _{2}(x)\rightarrow 1$ as $x\rightarrow
0$ ensures that for each $\epsilon >0$ there is some $c_{0}=c_{0}(\epsilon )$
such that 
\begin{equation*}
\left\vert \frac{\Phi _{1}(R)}{\Phi _{2}(R)}-1\right\vert \leq \epsilon 
\text{ for }R\leq c_{0}.
\end{equation*}%
Thus (\ref{upper}) holds for all $R\leq \min (c_{0},c_{1})$ (where $c_{1}$
was introduced in Lemma \ref{lemma:cont}) and that implies $\overline{\dim }%
_{\Phi _{2}}E\leq (\overline{\dim }_{\Phi _{1}}E+\epsilon )(1+\epsilon )$
for any $\epsilon >0.$ Hence $\overline{\dim }_{\Phi _{2}}E\leq \overline{%
\dim }_{\Phi _{1}}E$. Since the roles of $\Phi _{1}$ and $\Phi _{2}$ can be
interchanged because the condition is symmetric, the opposite inequality
also holds.

The statement for the lower dimensions follows in the same way. Observe that
by symmetry there is no need to consider separately the hypothetical case
when one of the dimensions is zero.

(ii) Given $\epsilon >0$ choose $\delta >0$ such that 
\begin{equation*}
\max (|g(t_{0})/g(t)-1|,|g(t)/g(t_{0})-1|)\leq \epsilon
\end{equation*}%
for any $|t-t_{0}|<\delta $. For each such $t$, apply Lemma \ref{lemma:cont}
to $\Phi _{t}$ and $\Phi _{t_{0}}$.
\end{proof}

In Proposition \ref{ContFailure} we will see that the convergence of (ii)
need not hold if $g(t_{0})=0$.

\begin{remark}
In order to apply Lemma \ref{lemma:cont} to prove the `continuity' property
of Prop. \ref{Propo-equal}(ii) , it was necessary that the convergence of $%
\Phi _{t}(x)/\Phi _{t_{0}}(x)$ was uniform in $x$. For instance, Lemma \ref%
{lemma:cont} cannot be applied to families such as $\Phi _{t}(x)=|\log
x|^{-t}$ with $t\in (0,1]$. We do not know if there is a one-parameter
family of dimension functions which range continuously from the
quasi-Assouad to the Assouad dimensions.
\end{remark}

The quasi-Assouad dimensions can also be understood as special cases of $%
\Phi $-dimensions, but the functions $\Phi $ need to be tailored for the
specific set.

\begin{proposition}
\label{qA}For any $E$ $\subseteq X,$ there are dimension functions $\Phi
_{1},$ $\Phi _{2}$ (depending on $E),$ which tend to $0$ and satisfying $%
\dim _{qA}E=\overline{\dim }_{\Phi _{1}}E$ and $\dim _{qL}E$ $=\underline{%
\dim }_{\Phi _{2}}E.$
\end{proposition}

\begin{proof}
Since our choice of dimension functions will satisfy $\Phi _{i}\rightarrow
0, $ it will automatically be true that $\overline{\dim }_{\Phi _{1}}E\geq
\dim _{qA}E$ and $\underline{\dim }_{\Phi _{2}}E$ $\leq \dim _{qL}E$. Thus
we need only check the opposite inequalities.

First, consider the quasi-Assouad dimension. Put $d=\dim _{qA}E$. By
definition, for each $n\in \mathbb{N}$, there are $\delta _{n}\downarrow 0$
and $\rho _{n}\downarrow 0$ such that for all $r\leq R^{1+\delta _{n}}\leq
R\leq \rho _{n}$ and $z\in E,$ we have%
\begin{equation*}
N_{r}(B(z,R)\cap E)\leq \left( \frac{R}{r}\right) ^{d+1/n}.
\end{equation*}

Put $n_{1}=1$ and inductively define a subsequence $\{n_{j}\}$ so that $%
p_{j}^{1+\varepsilon _{j-1}}=p_{j+1}^{1+\varepsilon _{j}}$ where, for
notational convenience, we put $p_{j}=\rho _{n_{j}}$ and $\varepsilon
_{j}=\delta _{n_{j}}$. Since $\varepsilon _{j-1}>\varepsilon _{j}$, we also
have $p_{j+1}^{1+\varepsilon _{j-1}}<p_{j+1}^{1+\varepsilon _{j}},$ hence
there is some $q_{j}\in (p_{j+1},p_{j}]$ with $q_{j}^{1+\varepsilon
_{j-1}}=p_{j+1}^{1+\varepsilon _{j}}$.

We are now ready to define $\Phi _{1}$. For $R\in (q_{j},p_{j}]$ we put $%
\Phi _{1}(R)=\varepsilon _{j-1},$ while for $R\in (p_{j+1},q_{j}]$ we define 
$\Phi _{1}(R)$ by the rule that $R^{1+\Phi _{1}(R)}=p_{j+1}^{1+\varepsilon
_{j}}$. Observe that in either case, $\Phi _{1}(R)\geq \varepsilon _{j}$. It
is straight forward to verify that the function $R^{1+\Phi (R)}$ decreases
to $0$ as $R$ decreases to $0,$ and therefore $\Phi _{1}$ is a dimension
function.

If $R\in (p_{j+1},q_{j}],$ then $R^{1+\varepsilon _{j-1}}\leq
q_{j}^{1+\varepsilon _{j-1}}=R^{1+\Phi _{1}(R)}$ and thus $\Phi _{1}(R)\leq
\varepsilon _{j-1}$. This shows that $\Phi _{1}$ tends to $0$.

We will check that $\overline{\dim }_{\Phi _{1}}E\leq d+1/n_{N}$ for any
given $N$. To do this, choose any $R\leq p_{N}$, $r\leq R^{1+\Phi _{1}(R)}$
and $z\in E$. Find $k\geq N$ so $R\in (p_{k+1},p_{k}]$ so that $\Phi
_{1}(R)\geq \varepsilon _{k}$. Thus $r\leq R^{1+\varepsilon
_{k}}=R^{1+\delta _{n_{k}}}$ and consequently, since $n_{k}\leq n_{N}$, 
\begin{equation*}
N_{r}(B(z,R)\cap E)\leq \left( \frac{R}{r}\right) ^{d+1/n_{k}}\leq \left( 
\frac{R}{r}\right) ^{d+1/n_{N}},
\end{equation*}%
completing the verification.

The argument for the quasi-lower Assouad dimension is the same, building $%
\Phi _{2}$ using the fact that for each $n\in \mathbb{N}$, there are $\delta
_{n}\downarrow 0$ and $\rho _{n}\downarrow 0$ such that for all $r\leq
R^{1+\delta _{n}}\leq R\leq \rho _{n}$ and $z\in E,$ we have%
\begin{equation*}
N_{r}(B(z,R)\cap E)\geq \left( \frac{R}{r}\right) ^{d+1/n},
\end{equation*}%
and the proposition follows.
\end{proof}



Although the lower $\theta $-Assouad spectrum is bounded above by the lower
box dimension, it is not always bounded above by the Hausdorff dimension. In
fact, $\sup_{\theta }\underline{\dim }_{L}^{\theta }E$ is not even bounded
above by a uniform multiple of the Hausdorff dimension. This is a
consequence of \cite[Theorem 3.3]{FYInd} and \cite{CWC}; see particularly
the comments in \cite{FYInd} following the statement of the theorem.

However, we have the following relationship between the lower $\theta $%
-Assouad spectrum and the Hausdorff dimension that does not seem to have
been previously observed. This follows easily from \cite{GH} where it was
shown that $\dim _{qL}E\leq \dim _{H}E$ for closed subsets $E\subseteq 
\mathbb{R}^{d},$ but the same proof is valid for closed subsets in a
doubling metric space.

\begin{proposition}
\label{Propo-qLHau} If $E$ is a closed subset of $X$, then $\underline{\dim }%
_{L}^{\theta }E\leq \frac{1}{\theta }\dim _{H}E$ for any $\theta \in (0,1)$.
\end{proposition}

\begin{proof}
Fix $\theta \in (0,1)$ and recall that $\underline{\dim }_{\Phi _{\theta }}E=%
\underline{\dim }_{L}^{\theta }E$ for $\Phi _{\theta }=1/\theta -1$. If $%
\underline{\dim }_{\Phi _{\theta }}E=0$ there is nothing to prove, so assume 
$\alpha <\underline{\dim }_{\Phi _{\theta }}E$ for some $\alpha >0$.

The doubling property implies that the covering numbers, $N_{r},$ can be
replaced by the packing numbers, $P_{r},$ in the definition of the $\Phi $%
-dimensions. Thus we can pick $\rho =\rho (\theta ,\alpha )>0$ such that for
any $z\in E$ and any $r\leq 2R^{1/\theta }\leq R\leq \rho ,$ 
\begin{equation*}
P_{r}(B(z,R)\cap E)\geq 2^{\alpha }(R/r)^{\alpha }.
\end{equation*}%
In particular, $P_{2R^{1/\theta }}(B(z,R)\cap E)\geq R^{(1-1/\theta )\alpha
} $.

In \cite[Proposition 10]{GH}, it is shown that under this assumption, there
is a probability measure $\mu ,$ supported on $E,$ and a constant $A$ such
that 
\begin{equation*}
\mu (U)\leq A(\text{diam}(U))^{\alpha \theta }
\end{equation*}%
for all Borel sets $U$. But then the mass distribution principle (see \cite[%
Proposition 2.1]{Fal}) implies $\theta \alpha \leq \dim _{H}E$. As this is
true for all $\alpha <\underline{\dim }_{\Phi _{\theta }}E,$ it must be that 
$\theta \underline{\dim }_{\Phi _{\theta }}E\leq \dim _{H}E$.
\end{proof}

\begin{corollary}
If $E$ is a closed set and $\Phi (x)\rightarrow \delta >0$ as $x\rightarrow
0,$ then $\underline{\dim }_{\Phi }E\leq (\delta +1)\dim _{H}E$.
\end{corollary}

\subsection{Dimensions of decreasing sequences\label{DecSeq}}

In \cite{GH} and \cite{GHM} it was shown that if $E=\{x_{n}\}\subseteq 
\mathbb{R}^{+}$ is a decreasing sequence with the sequence of `gaps', $%
\{x_{n}-x_{n+1}\},$ also decreasing, then both the upper Assouad and upper quasi-Assouad
dimensions of $E$ are either $0$ or $1$. The upper Assouad dimension of such a set
is $0$ if and only if the sequence of gaps is lacunary. Likewise, the
upper quasi-Assouad dimension is $0$ if and only if $\overline{\dim }_{B}E=0$.

This dichotomy fails for the upper $\Phi $-dimensions. Indeed, if we choose $\Phi
(x)=\delta >0$ for all $x$, it follows from \cite{FTale} and \cite[Theorem
6.2]{FYAdv}, that 
\begin{equation*}
\overline{\dim }_{\Phi }E=\min \{(1+\delta ^{-1})\overline{\dim }_{B}E,1\}.
\end{equation*}%
Therefore, $\overline{\dim }_{\Phi }E\in (0,1)$ if $0<\overline{\dim }%
_{B}E<\delta /(1+\delta )$. However, in this case the upper $\Phi $-dimension is
bounded above by the upper quasi-Assouad dimension, and necessarily $\dim
_{qA}E=\dim _{A}E=1$.

More interestingly, the dichotomy fails also for upper $\Phi $-dimensions that lie
between the upper quasi-Assouad and upper Assouad dimensions. Indeed, we have the
following.

\medskip

\begin{example}
\label{Ex:Dec} There is a decreasing set $E$, with decreasing gaps, and a
dimension function $\Phi \rightarrow 0,$ with $\dim _{qA}E=0$ and $\dim
_{A}E=1,$ but with $0<\overline{\dim }_{\Phi }E<1$.
\end{example}

\begin{proof}[Construction]
Let $E=\{x_{n}\}_{n=1}^{\infty }$ where $x_{n}=n^{-\log n}$. Define $\Phi $
by the rule $x_{n}^{1+\Phi (x_{n})}=2(4n)^{-(1+\log 4n)}\log (4n)$ and
extend $\Phi $ to $\mathbb{R}$ by setting $\Phi (x)=\Phi (x_{n})$ if $x\in
(x_{n+1},x_{n})$. We will verify that $E$ and $\Phi $ have the stated
properties. Of course, if $\Phi (x)\rightarrow 0$ as $x\rightarrow 0,$ then
we must have $\dim _{qA}E\leq \overline{\dim }_{\Phi }E\leq $ $\dim _{A}E$
and thus the properties $\dim _{qA}E=0$ and $\dim _{A}E=1$ will follow once
we have shown that $\Phi $ tends to $0$ and $0<\overline{\dim }_{\Phi }E<1$.

The fact that $E$ is a decreasing set follows from the fact that the
function $f(z)=z^{-\log z}$ has negative derivative. Similarly, $x^{1+\Phi
(x)}$ can be seen to be decreasing by checking the function $%
g(z)=z^{-(1+\log z)}\log z$ has negative derivative for large $z$. Thus $%
\Phi $ is a dimension function. One can directly calculate $\Phi $ and see
that $\Phi (x_{n})\sim 1/\log n$. That shows $\Phi (x)\rightarrow 0$ as $%
x\rightarrow 0$.

From the derivative of the function $h(x)=x^{-\log x}-(x+1)^{-\log (x+1)}$
one can also confirm that the sequence $\{x_{n}-x_{n+1}\}$ is decreasing. An
application of the mean value theorem shows that $x_{n}-x_{n+1}$ $%
=-f^{\prime }(\xi _{n})$ for some $\xi _{n}\in \lbrack n,n+1]$ and thus 
\begin{equation*}
2(n+1)^{-\log (n+1)}\log (n+1)/(n+1)\leq x_{n}-x_{n+1}\leq 2n^{-\log n}\log
n/n.
\end{equation*}%
This shows that if we take $R=x_{k}$ and put 
\begin{equation*}
r=R^{1+\Phi (R)}=2(4k)^{-(1+\log 4k)}\log (4k),
\end{equation*}%
then 
\begin{equation*}
x_{4k}-x_{4k+1}\leq r\leq x_{4k-1}-x_{4k}.
\end{equation*}%
Because the gaps are decreasing in length, $x_{i}-x_{i+1}\geq r$ whenever $%
i=k,\ldots,4k-1,$ and $x_{i}=x_{i+1}\leq r$ whenever $i\geq 4k$.
Consequently, 
\begin{equation*}
N_{r/2}\left( B(0,R)\cap E\right) =3k+\frac{x_{4k}}{r}=3k+\frac{(4k)^{-\log
4k}k}{2(4k)^{-\log 4k}\log 4k}
\end{equation*}%
and hence for large enough $k$, 
\begin{equation*}
3k\leq N_{r/2}\left( B(0,R)\cap E\right) \leq 4k
\end{equation*}%
Since%
\begin{equation*}
\frac{R}{r}=\frac{k^{1+2\log 4}}{2\log 4k}4^{1+\log 4},
\end{equation*}%
we deduce that $\overline{\dim }_{\Phi }E\geq 1/(1+2\log 4)$.

A similar statement holds for any $r$ with $x_{4k}-x_{4k+1}\leq r\leq
x_{4k-1}-x_{4k}$. More generally, there are constants $c_{1},c_{2},c_{3}>0$
such that if%
\begin{equation*}
x_{4n}-x_{4n+1}\leq r\leq x_{4n-1}-x_{4n}
\end{equation*}%
where $n=Lk+j$ with $0\leq j<k$ and $L\geq 4,$ then 
\begin{equation*}
N_{r/2}\left( B(0,R)\cap E\right) =n-k+\frac{x_{n}}{r}\leq c_{1}(n-k+\frac{n%
}{\log n})\leq c_{2}n,
\end{equation*}%
and 
\begin{equation*}
\frac{R}{r}\geq c_{3}\frac{n^{1+\log n}}{k^{\log k}\log n}.
\end{equation*}%
Thus, for $t=(1+\log 3)^{-1}$, we get 
\begin{equation*}
\left( \frac{R}{r}\right) ^{t}\geq n\left( c_{3}\frac{n^{(\log n-\log 3)}}{%
k^{\log k}\log n}\right) ^{t}\geq n\left( c_{3}\frac{L^{\log Lk-\log
3}k^{\log L-\log 3}}{\log ((L+1)k)}\right) ^{t}
\end{equation*}%
and the last quotient is bounded away from $0$.

If $R\in (x_{k+1},x_{k}),$ then since $\Phi (R)=\Phi (x_{k})$ we make a
similar argument. Finally, we note that if $z>0,$ then the decreasingness of
the gaps means 
\begin{equation*}
N_{r/2}\left( B(z,R)\cap E\right) \leq N_{r/2}(B(0,R)).
\end{equation*}%
Thus $0<1/(1+2\log 4)\leq \overline{\dim }_{\Phi }E\leq 1/(1+\log 3)<1$.
\end{proof}

\begin{remark}
It is an open problem to characterize the dimension functions $\Phi $ for
which the $0,1$ dichotomy holds.
\end{remark}

\section{Examples of $\Phi $-dimensions\label{Examples}}

In this section we will construct various examples. These will show the
sharpness of some of the basic properties, such as Proposition \ref{DimEqual}%
, as well as illustrating their distinctness. In particular, we will give an
example of a set with specified values for a countable family of $\Phi $%
-dimensions and whose set of all dimensions between quasi-Assouad and
Assouad is an interval.

In all these examples, the set $E$ will be a Cantor set, by which we mean a
perfect subset of $[0,1]$ of Lebesgue measure zero, that has a construction
as outlined below. We begin this section by determining a formula for the $%
\Phi $-dimension of Cantor sets. It will be convenient to make use of the
following notation.

\begin{notation}
\label{sim}We write $f\sim g$ if there are positive constants $c_{1},c_{2}$
such that $c_{1}f(x)\leq g(x)\leq c_{2}f(x)$ for all $x$. The symbols $%
\gtrsim $ and $\lesssim $ are defined similarly.
\end{notation}

\subsection{$\Phi $-dimensions of Cantor sets}

Given a decreasing, summable sequence, $a=\{a_{j}\}$ with $\sum_{j}a_{j}=1,$
by the \textbf{Cantor set associated with }$a$, denoted by $C_{a},$ we mean
the compact subset of $[0,1]$ constructed as follows: In the first step, we
remove from $[0,1]$ an open interval of length $a_{1}$, resulting in two
closed intervals $I_{1}^{1}$ and $I_{2}^{1}$. Having constructed the $k$-th
step, we obtain the closed intervals $I_{1}^{k},\ldots ,I_{2^{k}}^{k}$
contained in $[0,1]$. The intervals $I_{j}^{k},$ $j=1, \ldots,2^{k},$ are
called the Cantor intervals of step $k$. The next step consists in removing
from each $I_{j}^{k}$ an open interval of length $a_{2^{k}+j-1}$, obtaining
the closed intervals $I_{2j-1}^{k+1}$ and $I_{2j}^{k+1}$. We define 
\begin{equation*}
C_{a}:=\bigcap_{k\geq 1}\bigcup_{j=1}^{2^{k}}I_{j}^{k}.
\end{equation*}%
This construction uniquely determines the set because the lengths of the
removed intervals on each side of a given gap are known. The classical
middle-third Cantor set is the Cantor set associated with the sequence $%
\{a_{i}\}$ where $a_{i}=3^{-n}$ if $2^{n-1}\leq i\leq 2^{n}-1$. All
associated Cantor sets are uncountable, compact, totally disconnected and,
in fact, are all homeomorphic.

If we put 
\begin{equation*}
s_{n}=2^{-n}\sum_{j\geq 2^{n}}a_{j},
\end{equation*}%
then $s_{n}$ is the average length of the Cantor intervals of step $n$. The
decreasing property of the sequence $\{a_{j}\}$ ensures that all the
intervals of step $n$ have lengths satisfying

\begin{equation*}
s_{n+1}\leq \text{length}(I_{j}^{n})\leq s_{n-1}
\end{equation*}%
and that $s_{n}\geq a_{2^{n+1}}$. Of course, always $s_{n+1}\leq s_{n}/2$.

When the gap sizes $a_{2^{n}}=\cdot \cdot \cdot =a_{2^{n+1}-1}$ for all $n$,
the intervals at step $n$ all have the same length (namely $s_{n}$), and the
Cantor set is sometimes called a central Cantor set. The classical
middle-third Cantor set is such an example. In this case, the ratio $%
s_{j+1}/s_{j}$ is referred to as the ratio of dissection at step (or level) $%
j$.

We will assume the sequence $\{a_{j}\}$ is \textbf{doubling}, meaning there
is a constant $\kappa $ such that $a_{n}\leq \kappa a_{2n}$ for all $n$.
This ensures that%
\begin{equation*}
\tau =\inf s_{n+1}/s_{n}>0
\end{equation*}%
since%
\begin{equation*}
s_{n}\leq \frac{1}{2^{n}}\left( \sum_{j=2^{n+1}}^{\infty
}a_{j}+a_{2^{n}}2^{n}\right) \leq 2s_{n+1}+\kappa ^{2}a_{2^{n+2}}\leq
(2+\kappa ^{2})s_{n+1}.
\end{equation*}%
Thus under the doubling assumption we have $s_{j}\sim s_{j+1}$. It is easily
seen that such a Cantor set is uniformly perfect, where we recall that a set 
$E$ is called uniformly perfect if there is a constant $c>0$ so that for
every $z\in E$ and $r>0$ we have $B(z,r)\diagdown B(z,cr)\neq \emptyset $
whenever $E\diagdown B(z,r)\neq \emptyset $. A set $E$ is uniformly perfect
if and only if $\underline{\dim }_{L}E>0$ \cite{KLV}, and consequently, $%
\underline{\dim }_{\Phi }C_{a}>0$ whenever $a$ is a doubling sequence.

For Cantor sets, it is helpful to understand the comparison $r\leq R^{1+\Phi
(R)}$ in terms of the sequence $\{s_{n}\}$. For this we introduce the
following notation.

\begin{notation}
Given $\Phi \in \mathcal{D}$ and a doubling, decreasing, summable sequence $%
a=\{a_{j}\},$ define the associated \textbf{depth function} $\phi :\mathbb{%
N\rightarrow N}$ by the rule that $\phi (n)$ is the minimal integer $j$ such
that $s_{n+j}\leq s_{n}^{1+\Phi (s_{n})}$.
\end{notation}

In other words, $\phi (n)$ is the minimal integer with $s_{n+\phi
(n)}/s_{n}\leq s_{n}^{\Phi (s_{n})}$. We remark that $\phi $ depends on both 
$\Phi $ and the underlying Cantor set (or, equivalently, the sequence $%
\{a_{j}\}$). We will frequently refer to $\Phi /\phi $ as a \textbf{%
dimension/depth function pair} associated with the Cantor set.

If $\phi $ is bounded, then the sequence $\{s_{n}^{\Phi (s_{n})}\}$ is
bounded away from $0$. The decreasingness of the function $R^{1+\Phi (R)}$
implies that if $s_{n}\leq R\leq s_{n-1}$, then 
\begin{equation*}
\tau s_{n}^{\Phi (s_{n})}\leq R^{\Phi (R)}\leq \frac{1}{\tau }s_{n-1}^{\Phi
(s_{n-1})}.
\end{equation*}%
Hence if $\phi $ is bounded, then Proposition \ref{ADim} implies the upper
(or lower) $\Phi $-dimension coincides with the upper (resp., lower) Assouad
dimension.

A very useful observation for constructing examples is to note that if $E$
is \textit{any} Cantor set with $\tau =$ $\inf s_{j+1}/s_{j}$ and $\rho
=\sup s_{j+1}/s_{j}\leq 1/2$, and $\Phi $ is a dimension function with
associated depth function $\phi $ with respect to $E$, then we have%
\begin{equation}
\frac{(\phi (n)-1)}{n}\frac{\log \rho }{\log \tau }\leq \Phi (s_{n})\leq 
\frac{\phi (n)}{n}\frac{\log \tau }{\log \rho }.  \label{PreDepth}
\end{equation}%
This is because the doubling property ensures 
\begin{equation*}
\tau ^{\phi (n)}\leq \frac{s_{n+\phi (n)}}{s_{n}}\leq s_{n}^{\Phi
(s_{n})}\leq \rho^{n\Phi (s_{n})} \text{ \ and \ } \tau^{n\Phi (s_{n})}\le
s_{n}^{\Phi (s_{n})}\le \frac{s_{n+\phi (n)-1}}{s_{n}}\le \rho^{\phi (n)-1}.
\end{equation*}

If, in addition, $\phi (n)\geq 2$ (as is typically the case in interesting
examples), then we see that $\phi (n)$ is comparable to $n\Phi (s_{n})$ with
constants depending only on $\tau ,\rho $ because

\begin{equation}
\frac{\phi (n)}{n}\frac{\log \rho }{2\log \tau }\leq \Phi (s_{n})\leq \frac{%
\phi (n)}{n}\frac{\log \tau }{\log \rho }  \label{depth/dim}
\end{equation}

\begin{remark}
Notice that if we are given an increasing function $\phi :$ $\mathbb{%
N\rightarrow N}$, and a Cantor set $C_{a},$ we can define a function $\Phi $
by the rule $R^{1+\Phi (R)}=s_{n+\phi (n)}$ if $R\in (s_{n+1},s_{n}]$. If $%
R_{1}\leq R_{2}$ with $R_{1}\in (s_{n+1},s_{n}]$ and $R_{2}\in
(s_{k+1},s_{k}]$, then $n\geq k,$ so $\phi (n)\geq \phi (k)$ and hence $%
s_{n+\phi (n)}\leq s_{k+\phi (k)}$. Consequently, $R_{1}^{1+\Phi
(R_{1})}=s_{n+\phi (n)}\leq s_{k+\phi (k)}\leq R_{2}^{1+\Phi (R_{2})}$.
Furthermore, $R^{1+\Phi (R)}=s_{n+\phi (n)}\rightarrow 0$ as $n\rightarrow
\infty $ and hence as $R\rightarrow 0$. Thus $\Phi $ is a dimension function
with associated depth function $\phi $.
\end{remark}

\begin{corollary}
(i) If $\phi (n)/n\rightarrow \infty $, then $\overline{\dim }_{\Phi }C_{a}=%
\overline{\dim }_{B}C_{a}$ and $\underline{\dim}_{\Phi} C_a=\underline{\dim}%
_B C_a$.

(ii) The quasi-Assouad dimensions are obtained by taking $\phi (n)=\delta n$
and letting $\delta \rightarrow 0$.
\end{corollary}

\begin{proof}
These follow from the fact that $\phi (n)/n\sim \Phi (s_{n})$. The statement
in (i) about the lower $\Phi$ dimension follows from Proposition \ref%
{Prop:box} (ii), since $\underline{\dim}_{\Phi} C_a > 0$ because $C_a$ is
uniformly perfect.
\end{proof}

More generally, we have the following formulas for the $\Phi $-dimensions of
Cantor sets.

\begin{theorem}
\label{Cantorformula}Let $a$ be a decreasing, summable, doubling sequence
and $C_{a}$ the associated Cantor set. The upper and lower $\Phi $%
-dimensions of $C_{a}$ are given by 
\begin{equation}
\overline{\dim }_{\Phi }C_{a}=\inf \Bigl\{\beta :(\exists \ k_{0},c_{0}>0)\
(\forall k\geq k_{0}\text{, }n\geq \phi (k))\text{ }\left( \frac{s_{k}}{%
s_{k+n}}\right) ^{\beta }\geq c_{0}2^{n}\Bigr\}  \label{Thetadim}
\end{equation}%
and%
\begin{equation}
\underline{\dim }_{\Phi }C_{a}=\sup \Bigl\{\beta :(\exists \ k_{0},c_{0}>0)\
(\forall k\geq k_{0}\text{, }n\geq \phi (k))\text{ }\left( \frac{s_{k}}{%
s_{k+n}}\right) ^{\beta }\leq c_{0}2^{n}\Bigr\}.  \label{LowerDim}
\end{equation}
\end{theorem}

The proof is omitted as the arguments are similar to those given in \cite%
{GHM} for Assouad dimensions and in \cite{CWC} and \cite{LX} for the
quasi-Assouad dimensions.

\subsection{Basic properties revisited}

With the formulas for the $\Phi $-dimensions of Cantor sets, it is easy to
give examples of sets with any specified $\Phi $-dimension in $(0,1)$. The
key idea is that if $E$ is a central Cantor set with ratios of dissection $%
r_{k}$ at step $k,$ and there is an increasing sequence of integers $%
\{n_{j}\}$ (possibly even very sparse) such that $r_{k}=\rho $ for all $%
k=n_{j}+1,\ldots,n_{j}+\phi (n_{j})$ and $r_{k}=\tau \leq \rho $ otherwise,
then $\overline{\dim }_{\Phi }E=\log 2/\left\vert \log \rho \right\vert ,$
where $\Phi /\phi $ is a dimension/depth function pair associated with $E$.
A similar idea can be applied for the lower $\Phi $-dimension.

In this subsection we will use this principle to obtain (partial) converses
to Proposition \ref{DimEqual}. First, we will show that the continuity
properties described in Proposition \ref{DimEqual}(ii) can fail when $%
g(t_{0})=0$.

\begin{proposition}
\label{ContFailure}Suppose $\phi $ is an increasing depth function tending
to infinity, but with $\phi (n)/n\rightarrow 0$ as $n\rightarrow \infty $.
There is a central Cantor set $E$ such that if $\Phi $ is the dimension
function associated with the depth function $\phi $ (and Cantor set $E$) and 
$\Phi _{t}=t\Phi ,$ then $\dim _{\Phi _{t}}E\in \lbrack \dim _{qA}E,\dim
_{A}E]$, but $\lim_{t\rightarrow 0}\overline{\dim }_{\Phi _{t}}E<\dim _{A}E.$
\end{proposition}

\begin{proof}
Choose $A,B>0$ from (\ref{PreDepth}) such that if $E$ is a Cantor set with $%
\inf s_{j+1}/s_{j}\geq 1/27$ and $\Psi /\psi $ is any dimension/depth
function pair associated with $E,$ then 
\begin{equation*}
A\frac{\psi (n)-1}{n}\leq \Psi (s_{n})\leq B\frac{\psi (n)}{n}\text{ for all 
}n\text{.}
\end{equation*}%
In particular, this holds for the depth function $\phi $ and any associated
dimension function $\Phi $, and also for the depth function $\phi _{t}$
associated with $t\Phi $. Without loss of generality we can assume $\phi
(n)\geq 2$ for all $n$ and therefore for all $k\in \mathbb{N}$ 
\begin{equation*}
A\frac{\phi (n)}{2kn}\leq \frac{1}{k}\Phi (s_{n})=\Phi _{1/k}(s_{n})\leq B%
\frac{\phi _{1/k}(n)}{n}.
\end{equation*}%
That shows that for each $k$ there is some $N_{k}$ such that if $n\geq N_{k}$%
, then $\phi _{1/k}(n)\geq 2$. Thus we also have $A\phi _{1/k}(n)/(2n)\leq
\Phi _{1/k}(s_{n})$ for all $n\geq N_{k}$ and therefore with the constant $%
C=A/(2B)$ (and any such Cantor set $E$) we have 
\begin{equation*}
\frac{C}{k}\phi (n)\leq \phi _{1/k}(n)\leq \frac{1}{Ck}\phi (n)\text{ for
all }k\text{ and }n\geq N_{k}\text{.}
\end{equation*}

To construct the Cantor set $E$, we will first choose an integer-valued
function $f(n)\rightarrow \infty $ with $f(n)/\phi (n)\rightarrow 0$ as $%
n\rightarrow \infty $. Then choose an increasing sequence of integers $%
\{n_{k}\}$ with $n_{k}\geq \max (2N_{k},8n_{k-1})$ and satisfying 
\begin{equation*}
f(n_{k})\leq \min \left( \frac{n_{k}}{8},\text{ }\frac{C}{2k}\phi
(n_{k}-f(n_{k}))\right) .
\end{equation*}%
The Cantor set will be defined by setting the ratios of dissection to be $%
1/3 $ on steps $n_{j}+1,\ldots,n_{j}+f(n_{j})$ for all $j=1,2,\ldots$ and
equal to $1/27$ on all other levels. Certainly, $\dim _{A}E=\log 2/\log 3$.

Let $r_{i}$ denote the ratio of dissection at step $i$. Our choice of $n_{j}$
ensures that if $n\in \{n_{j}-f(n_{j})+1,\ldots,n_{j}+f(n_{j})\}$ and $m\geq
2f(n_{j}),$ then at least as many $r_{i}=1/27$ as are equal to $1/3$ for $i$
ranging over $\{n+1,\ldots,n+m\}$. Hence the geometric mean of these ratios
is at most $1/9$. The same conclusion clearly also holds if $n\notin
\{n_{j}-f(n_{j})+1,\ldots,n_{j}+f(n_{j})\}$.

In order to bound $\overline{\dim }_{\Phi _{1/k}}E$ we use formula (\ref%
{Thetadim}), noting first that it suffices to consider $\left(
s_{n}/s_{n+m}\right) ^{1/m}$ where $n\geq n_{k}$ and $m\geq \phi _{1/k}(n)$.
If $n\in \{n_{j}-f(n_{j})+1,\ldots,n_{j}+f(n_{j})\}$ for some $j\geq k$,
then as $\phi $ is increasing and $n\geq N_{k}\,,$ 
\begin{equation*}
m\geq \phi _{1/k}(n)\geq \frac{C}{k}\phi (n)\geq \frac{C}{k}\phi
(n_{j}-f(n_{j}))\geq \frac{C}{j}\phi (n_{j}-f(n_{j}))>2f(n_{j})\text{.}
\end{equation*}%
By our previous remark, $\left( s_{n+m}/s_{n}\right) ^{1/m}\leq 1/9$. The
same bound clearly holds if $n\geq n_{k}$ does not belong to any such
interval. Consequently, (\ref{Thetadim}) implies $\overline{\dim }_{\Phi
_{1/k}}E\leq \log 2/\log 9$. By monotonicity, $\overline{\dim }_{\Phi
_{t}}E\leq \log 2/\log 9$ for all $t>0.$
\end{proof}

\begin{remark}
We remark that a similar argument could be used to prove that there is a
central Cantor set $E$ and dimension function $\Phi $ so that $%
\lim_{t\rightarrow 0}\underline{\dim }_{\Phi _{t}}E>\dim _{L}E$. One could
also similarly arrange for $\overline{\dim }_{\Phi _{t}}E\in \lbrack \dim
_{qA}E,\dim _{A}E]$, but $\lim_{t\rightarrow \infty }\overline{\dim }_{\Phi
_{t}}E>\dim _{qA}E$ and likewise for the quasi-lower Assouad dimension.
\end{remark}

We will use a similar technique to obtain a partial converse to Proposition %
\ref{DimEqual}.

\begin{theorem}
\label{Diff}Suppose $\Phi _{1},\Phi _{2}$ are dimension functions decreasing
to $0$ as $x\rightarrow 0$ with $\left\vert \log x\right\vert \Phi
_{2}(x)\rightarrow \infty $ as $x\rightarrow 0$. Assume there is some $\xi
>0 $ such that $\Phi _{1}(x)\geq (1+\xi )\Phi _{2}(x)$ for all $x$
sufficiently small. Then there is a Cantor set $E$ such that $\overline{\dim 
}_{\Phi _{1}}E<\overline{\dim }_{\Phi _{2}}E$ and a Cantor set $F$ with $%
\underline{\dim }_{\Phi _{1}}F>\underline{\dim }_{\Phi _{2}}F$.
\end{theorem}

\begin{proof}
We will give the proof for the upper $\Phi $-dimension. The lower $\Phi $%
-dimension case is similar.

The monotonicity property of the $\Phi $-dimensions implies that $%
\overline{\dim }_{\Phi _{1}}E\leq \overline{\dim }_{\Phi _{2}}E$ for all
sets $E$. It is the strictness of the inequality that we need to verify for
an appropriate choice of $E$.

The strategy of the proof will be to build a central Cantor set by
inductively specifying the ratios of dissection at each level. For most
levels, the ratio will be a fixed small number, say $\tau $. However, we
will specify the ratios to be a fixed number $\rho >\tau $ on the levels $%
n_{j}+1,\ldots,n_{j}+\phi _{2}(n_{j}),$ where $\phi _{2}$ is the depth
function associated with $\Phi _{2}$ and the Cantor set, and $\{n_{j}\}$ is
a sparse set. By consideration of $(s_{n_{j}+\phi
_{2}(n_{j})}/s_{n_{j}})^{1/\phi _{2}(n_{j})}$ (the geometric mean of the
ratios at levels $n_{j}+1,\ldots,n_{j}+\phi _{2}(n_{j})$) and the formula
for the $\Phi $-dimensions of Cantor sets from (\ref{Thetadim}), we have $%
\overline{\dim }_{\Phi _{2}}E=\log 2/|\log \rho |$. However, these depths
will be too shallow to give the $\Phi _{1}$-dimension and consequently we
will be able to conclude that $\overline{\dim }_{\Phi _{1}}E<\log 2/|\log
\rho |$.

One complication with this strategy is that the depth functions depend on
the construction of the Cantor set. However, our construction of the Cantor
set depends (at least, to some extent) on the depth functions. Fortunately,
we do have enough control on the depth functions to overcome this
complication. We address this issue first.

Fix small $\varepsilon >0$ such that 
\begin{equation*}
\left( \frac{1-\varepsilon }{1+\varepsilon }\right) ^{2}(1+\xi )\geq (1+\xi
/2).
\end{equation*}%
Choose $0<\tau <\rho <1/2$ with $\left\vert \log \tau /\log \rho \right\vert
\leq 1+\varepsilon $. It follows from (\ref{PreDepth}) that if $E$ is any
Cantor set with all ratios between $\tau $ and $\rho ,$ and $\Phi /\phi $
any dimension/depth function pair associated with $E$, then%
\begin{equation}
\frac{(\phi (n)-1)}{n}\frac{\log \rho }{\log \tau }\leq \Phi (s_{n})\leq 
\frac{\phi (n)}{n}\frac{\log \tau }{\log \rho }\leq (1+\varepsilon )\frac{%
\phi (n)}{n}.  \label{Phi}
\end{equation}%
By assumption, given any $C>0,$ there is some $x_{0}=x_{0}(C)$ such that if $%
x\leq x_{0},$ then $\left\vert \log x\right\vert \Phi _{i}(x)\geq C$. Choose 
$N_{0}$ such that $\tau ^{N_{0}}\leq x_{0}$. Since the functions $\Phi _{i}$
are decreasing as $x\rightarrow 0,$ it follows that if $n\geq N_{0}$ and $E$
is a Cantor set with all ratios of dissection at least $\tau ,$ then $\Phi
_{i}(s_{n})\geq \Phi _{i}(\tau ^{n})\geq C/\left\vert \log \tau
^{n}\right\vert $ and hence $n\Phi _{i}(s_{n})\geq (1+\varepsilon
)/\varepsilon $ if we take a suitable choice for $C,$ depending on $%
\varepsilon $ and $\tau $. Coupled with the right hand side of (\ref{Phi}),
this shows that for all $n\geq N_{0},$ 
\begin{equation*}
\phi _{i}(n)\geq \frac{n\Phi _{i}(s_{n})}{1+\varepsilon }\geq \frac{1}{%
\varepsilon }
\end{equation*}%
and hence $\phi _{i}(n)-1\geq (1-\varepsilon )\phi _{i}(n)$ for $i=1,2$.
Consequently, using the left hand side of (\ref{Phi}) we also have%
\begin{equation*}
\frac{\phi _{i}(n)}{n}\frac{(1-\varepsilon )}{(1+\varepsilon )}\leq \Phi
_{i}(s_{n})\text{ for all }n\geq N_{0}.
\end{equation*}%
As $\Phi _{i}\downarrow 0$, this further ensures that there exists $N_{1}$
such that 
\begin{equation*}
\phi _{i}(n)\leq \varepsilon n\text{ for }n\geq N_{1}.
\end{equation*}%
We remind the reader that having fixed $\varepsilon ,\tau $ and $\rho ,$
these inequalities and the choices of $N_{0}$ and $N_{1}$ depend only $\Phi
_{1}$ and $\Phi _{2}$ for \textit{any} choice of Cantor set, provided the
ratios of dissection are chosen from $[\tau ,\rho ]$. As we will see, these
relationships give us enough control on the depth functions.

\smallskip

\textbf{Construction of the Cantor Set: }

We will continue to use the notation from above. Let $n_{1}\geq \max
(8N_{0},8N_{1})$ and choose $n_{j+1}\geq 16n_{j}$. We will inductively
define a central Cantor set by specifying the ratios of dissection $r_{k}$
at each level $k$. To begin, we put $r_{k}=\tau $ for $k=1,\ldots,n_{1}$.
Thus $s_{n_{1}}=\tau ^{n_{1}}$. Define $\ell _{1}$ to be the least integer
with $\rho ^{\ell _{1}}\leq s_{n_{1}}^{\Phi _{2}(s_{n_{1}})}$ and let $%
r_{k}=\rho $ for $k=n_{1}+1,\ldots,n_{1}+\ell _{1}$. Notice that this
construction means $\ell _{1}=\phi _{2}(n_{1})\leq \varepsilon n_{1}$, thus $%
n_{1}+\ell _{1}<n_{2}$. We put $r_{k}=\tau $ for $k=n_{1}+\ell
_{1}+1,\ldots,n_{2}$.

Now we proceed inductively. We assume integers $\ell _{1},\ldots,\ell _{j-1}$
have been chosen in the same way and we have put $r_{k}=\rho $ if $%
k=n_{i}+1,\ldots,n_{i}+\ell _{i}$ for 
$i=1,\ldots,j-1,$ and $r_{k}=\tau $
otherwise on $\{1,\ldots,n_{j}\}$. Thus $s_{n_{j}}$ is determined. Define $%
\ell _{j}$ to be the least integer satisfying $\rho ^{\ell _{j}}\leq
s_{n_{j}}^{\Phi _{2}(s_{n_{j}})}$. We will put $r_{k}=\rho $ if $%
k=n_{j}+1,\ldots,n_{j}+\ell _{j}$ and $r_{k}=\tau $ on $\{n_{j}+\ell
_{j}+1,\ldots,n_{j+1}\}$. Again $\ell _{j}=\phi _{2}(n_{j})\leq \varepsilon
n_{j}$. This completes the construction of $E$.

\smallskip

\textbf{Verification of the }$\Phi _{i}$\textbf{-dimensions:}

The fact that the ratios equal $\rho $ on the consecutive levels $%
n_{j}+1,\ldots,n_{j}+\phi _{2}(n_{j})$ for all $j$ and are equal to $\tau $
otherwise, certainly means $\overline{\dim }_{\Phi _{2}}E=\log 2/\left\vert
\log \rho \right\vert .$

Since $\phi _{i}(n_{j})\leq \varepsilon n_{j}$ and $\Phi _{1}$ is
decreasing, the choice of $\varepsilon $ gives that for each $j$ and $n\in
\{n_{j}-\ell _{j},\ldots,n_{j}\},$ 
\begin{eqnarray}
\phi _{1}(n) &\geq &\frac{n}{1+\varepsilon }\Phi _{1}(s_{n})\geq \frac{%
n_{j}-\ell _{j}}{1+\varepsilon }\Phi _{1}(s_{n_{j}})  \notag \\
&\geq &\frac{(1-\varepsilon )n_{j}}{1+\varepsilon }\Phi _{1}(s_{n_{j}})\geq 
\frac{1-\varepsilon }{1+\varepsilon }(1+\xi )n_{j}\Phi _{2}(s_{n_{j}})
\label{2} \\
&\geq &\left( \frac{1-\varepsilon }{1+\varepsilon }\right) ^{2}(1+\xi )\phi
_{2}(n_{j})\geq (1+\xi /2)\phi _{2}(n_{j}).  \notag
\end{eqnarray}%
Since the sequence $\{s_{n}^{1+\Phi (s_{n})}\}$ is decreasing (for any
dimension function $\Phi $), for any $n$, $m\geq 1$ and associated depth
function $\phi $ we have%
\begin{equation*}
s_{n+m+\phi (n+m)}\leq s_{n+m}^{1+\Phi (s_{n+m})}\leq s_{n}^{1+\Phi
(s_{n})}<s_{n+\phi (n)-1}
\end{equation*}%
by the definition of $\phi $. That means $n+m+\phi (n+m)>n+\phi (n)-1$, and
as these are integers this implies, in particular, that for $i=1,2,$ 
\begin{equation}
n_{j}+m+\phi _{i}(n_{j}+m)\geq n_{j}+\phi _{i}(n_{j})  \label{Decrphi}
\end{equation}%
for all $m\geq 1$. As (\ref{2}) holds for $n=n_{j}$, this gives%
\begin{eqnarray}
n_{j}+m+\phi _{1}(n_{j}+m)-(n_{j}+\phi _{2}(n_{j})) &\geq &n_{j}+\phi
_{1}(n_{j})-(n_{j}+\phi _{2}(n_{j}))  \notag \\
&\geq &(\xi /2)\phi _{2}(n_{j}).  \label{3}
\end{eqnarray}

Since $\phi _{i}(n)\leq \varepsilon n$ for all $n\geq N_{1}$ and $\ell
_{j}=\phi _{2}(n_{j}),$ we also know that 
\begin{eqnarray*}
n_{j}+\phi _{2}(n_{j})+\max_{n\in \lbrack n_{j}-\ell _{j},n_{j}+\ell
_{j}]}\phi _{1}(n) &\leq &n_{j}+\varepsilon n_{j}+\varepsilon (n_{j}+\ell
_{j}) \\
&\leq &(1+\varepsilon )^{2}n_{j}\leq n_{j+1}/4 \\
&<&(n_{j+1}-\phi _{2}(n_{j+1}))/2.
\end{eqnarray*}%
In particular, this guarantees that if $n\in \{n_{j}-\ell
_{j}+1,\ldots,n_{j}+\ell _{j}\}$, then $n+\phi _{1}(n)<(n_{j+1}-\phi
_{2}(n_{j+1}))/2$. Together with (\ref{3}), it follows that for such $n$
there are at least $\left( \xi /2\right) \phi _{2}(n_{j})$ ratios equal to $%
\tau $ and at most $\phi _{2}(n_{j})=\ell _{j}$ ratios equal to $\rho $ on
the levels $n+1,\ldots,n+\phi _{1}(n)$. Hence the geometric mean of these
ratios is dominated by 
\begin{equation*}
\left( \rho ^{\ell _{j}}\tau ^{\xi \ell _{j}/2}\right) ^{1/((1+\xi /2)\ell
_{j})}=\rho ^{1/(1+\xi /2)}\tau ^{\xi /(2+\xi )}:=\sigma <\rho .
\end{equation*}

If $m\geq \phi _{1}(n)$, the choice of ratios ensures that there could only
be an even greater proportion of the ratios on the levels $n+1,\ldots,n+m$
having value $\tau $. Thus we can conclude that the geometric mean of the
ratios from the levels $n+1,\ldots,n+m$ is also dominated by $\sigma $
whenever $m\geq \phi _{1}(n)$ and $n\in \{n_{j}-\ell
_{j}+1,\ldots,n_{j}+\ell _{j}\}.$

If $n$ $\notin \{n_{j}-\ell _{j}+1,\ldots,n_{j}+\ell _{j}\}$ for any $j$,
then it is obvious from the construction that, on the levels $n+1,\ldots,n+m$
(for any $m\geq 1$), there are at least as many
ratios equal to $\tau $ as equal to $\rho $, and hence the geometric mean is even smaller.

We deduce that 
\begin{equation*}
\overline{\dim }_{\Phi _{1}}E\leq \frac{\log 2}{\left\vert \log \sigma
\right\vert }<\frac{\log 2}{\left\vert \log \rho \right\vert }=\overline{%
\dim }_{\Phi _{2}}E,
\end{equation*}
which concludes the proof.
\end{proof}

A modification of this argument would allow us to show that given $0<a<b<1/2$
there is an example of a Cantor set $E$ where 
\begin{equation*}
\overline{\dim }_{\Phi _{1}}E=\frac{\log 2}{\left\vert \log a\right\vert }<%
\frac{\log 2}{\left\vert \log b\right\vert }=\overline{\dim }_{\Phi _{2}}E.
\end{equation*}%
To do this, we will choose $0<c<a$. Then, instead of assigning ratio $\rho $
on the levels $n_{j}+1,\ldots,n_{j}+\phi _{2}(n_{j})$ and $\tau $ otherwise,
we will put ratios $b$ on levels $n_{2j}+1,\ldots,n_{2j}+\phi _{2}(n_{2j})$,
ratios $a$ on levels $n_{2j+1}+1,\ldots,n_{2j+1}+\phi _{1}(n_{2j+1})$ and
ratio $c$ elsewhere. The choice of sequence $\{n_{j}\}$ may need to be even
more sparse to ensure that $\phi _{1}(n_{j})$ is sufficiently large in
comparison with $\phi _{2}(n_{j})$ to guarantee that the geometric mean of
ratios from any $\phi _{1}(n)$ consecutive levels beginning at $n$ is at
most $a$. The fact that the ratios at levels $n_{2j+1}+1,\ldots,n_{2j+1}+%
\phi _{1}(n_{2j})$ are equal to $a$ implies that $\overline{\dim }_{\Phi
_{1}}E=\frac{\log 2}{\left\vert \log a\right\vert }$. From their values on
levels $n_{2j}+1,\ldots,n_{2j}+\phi _{2}(n_{2j})$ one can deduce that $%
\overline{\dim }_{\Phi _{2}}E=$ $\frac{\log 2}{\left\vert \log b\right\vert }
$. The details are left for the reader.

A further modification of the argument would also enable us to construct a
(single) Cantor set $E$ with both $\overline{\dim }_{\Phi _{1}}E<\overline{%
\dim }_{\Phi _{2}}E$ and $\underline{\dim }_{\Phi _{1}}E>\underline{\dim }%
_{\Phi _{2}}E$.

\subsection{Continuum of $\Phi $-dimensions}

In the next result we use the method described in the previous remark to
show that we can construct a Cantor set with countably many specified values
for $\Phi $-dimensions. Furthermore, there is a Cantor set with a continuum
of $\Phi $-dimensions between the quasi-Assouad and Assouad dimensions.

\begin{theorem}
\label{continuum}Assume that for each $p\in (0,1),$ $\Phi _{p}$ are
dimension functions decreasing to $0$ as $x\rightarrow 0$ and satisfying $%
\left\vert \log x\right\vert \Phi _{p}(x)\rightarrow \infty $ as $%
x\rightarrow 0$. Assume, also, that 
\begin{equation*}
\Phi _{p}(x)/\Phi _{q}(x)\rightarrow \infty \text{ as }x\rightarrow 0\text{
whenever }p>q.
\end{equation*}%
Choose any $0<\alpha <\beta <1$ and suppose $d:(0,1)\rightarrow \lbrack
\alpha ,\beta ]$ is monotonically decreasing and continuous. Then there is a
central Cantor set $E$ with 
\begin{equation*}
\overline{\dim }_{\Phi _{p}}E=d(p)\text{ for each }p\in (0,1)\text{.}
\end{equation*}

The analogous result holds for the lower $\Phi $-dimensions.
\end{theorem}

\begin{remark}
An example of a class of dimension functions that satisfy the conditions of
the theorem are the functions $\Phi _{p}(x)=|\log x|^{p-1}$.
\end{remark}

\begin{proof}
We will construct a central Cantor set $E$ with the property that if $%
f:(0,1)\bigcap \mathbb{Q}\rightarrow \lbrack a,b]$ is monotonically
decreasing, then $\overline{\dim }_{\Phi _{p}}E=\log 2/\left\vert \log
f(p)\right\vert $ for every rational $p\in (0,1)$. To obtain the theorem,
put $a=2^{-1/\alpha },$ $b=2^{-1/\beta }$ and define the decreasing
continuous function $f:(0,1)\rightarrow \lbrack a,b]$ by $f(x)=2^{-1/d(x)}$.
The proof follows from this property using the monotonicity of the functions 
$p\rightarrow \overline{\dim }_{\Phi _{p}}E$ and the fact that the function $%
d$ of the theorem is assumed to be continuous and decreasing.

As in the proof of the previous theorem our strategy will be to inductively
define the ratios of dissection of the Cantor set. These ratios will lie in $%
[a^{2},b]$ and so by (\ref{PreDepth}), with $c=\log b/(2\log a)$ we have%
\begin{equation*}
c(\phi (n)-1)\leq n\Phi (s_{n})\leq \frac{1}{c}\phi (n)\text{ for all }n%
\text{,}
\end{equation*}%
for any dimension function $\Phi $ and corresponding depth function $\phi $
associated with such a Cantor set.

Since $\left\vert \log x\right\vert \Phi _{p}(x)\rightarrow \infty $ for
each $p,$ there is a choice of $I_{p}\in \mathbb{N}$ such that if $n\geq
I_{p}$ and $x\leq a^{2n},$ then $\Phi _{p}(x)\geq C/\left\vert \log
x\right\vert $ for a suitable constant $C$. Consequently, as $s_{n}\geq
a^{2n}$, we will have $\phi _{p}(n)\geq cn\Phi _{p}(a^{2n})\geq 2$ for all $%
n\geq I_{p}$, (whatever the choice of $E,$ as long as the ratios lie between 
$a^{2}$ and $b$). Thus with $A=2/c$ and $B=c,$%
\begin{equation}
Bn\Phi _{p}(s_{n})\leq \phi _{p}(n)\leq An\Phi _{p}(s_{n})\text{ for all }%
n\geq I_{p}.  \label{4}
\end{equation}

As $\Phi _{p}$ decreases to $0$, there is also an index $J_{p}\in \mathbb{N}$
such that 
\begin{equation}
\Phi _{p}(b^{n})\leq 1/(8A)\text{ for all }n\geq J_{p}.  \label{5}
\end{equation}

As in the proof of Theorem \ref{Diff}, we will pick a sparse sequence $%
\{n_{j}\}$ and assign ratios $a^{2}$ except on the levels $%
\{n_{j}+1,\ldots,n_{j}+\phi _{r_{j}}(n_{j})\}$ where the ratios will be $%
f(r_{j}).$ Each $p$ must occur as an $r_{j}$ infinitely often so that we
will have $\overline{\dim }_{\Phi _{p}}E\geq \log 2/\left\vert \log
f(p)\right\vert .$ The numbers $n_{j}$ will need to be sufficiently sparse
so that if $q>p,$ this length of levels (where the ratio exceeds $f(q)$) is
too short to influence the $\overline{\dim }_{\Phi _{q}}E$ calculation.

\smallskip

\textbf{Construction of the Cantor set:}

To begin, we list $(0,1)\cap \mathbb{Q}$ as $\{r_{i}\}_{i=1}^{\infty }$
where each rational number is repeated infinitely often in $\{r_{i}\}$. To
start the construction of $E,$ pick $n_{1}\geq \max (I_{r_{1}},8J_{r_{1}})$.
We will set the ratios of dissection to be $a^{2}$ on the levels $%
\{1,\ldots,n_{1}\}$. Choose the minimal integer $\ell _{1}$ such that $%
f(r_{1})^{\ell _{1}}\leq s_{n_{1}}^{\Phi _{r_{1}}(s_{n_{1}})}$ and put $%
m_{1}=4(n_{1}+\ell _{1})$. Set the ratios equal to $f(r_{1})$ on the levels $%
\{n_{1}+1,\ldots,n_{1}+\ell _{1}\}$ and $a^{2}$ on the levels $\{n_{1}+\ell
_{1}+1,\ldots,m_{1}\}$.

Notice that $\ell _{1}=\phi _{r_{1}}(n_{1})$ and the choice of $n_{1}$
ensures that%
\begin{equation*}
\ell _{1}\leq An_{1}\Phi _{r_{1}}(s_{n_{1}})\leq An_{1}\Phi
_{r_{1}}(b^{n_{1}})\leq n_{1}/8
\end{equation*}%
by (\ref{4}) and (\ref{5}).

We proceed inductively and suppose we have chosen $n_{i},\ell _{i},m_{i}$
for $i=1,\ldots,j-1$, (with the properties described below) and have
specified that the ratios of dissection on levels $\{1,\ldots,m_{j-1}\}$
should be $a^{2}$ except on the levels $\{n_{i}+1,\ldots,n_{i}+\ell _{i}\},$
for $i=1,..,j-1,$ when they will be $f(r_{i}).$

Now pick $n_{j}$ large enough to satisfy the following conditions:

(i) $n_{j}\geq 8\max (I_{r_{j},}J_{r_{j}},m_{j-1}\}$ and

(ii) if $i<j$ and $r_{i}>r_{j},$ then 
\begin{equation*}
\Phi _{r_{i}}(s_{m_{j-1}}a^{2(n_{j}-m_{j-1})})\geq \frac{8A}{B}\Phi
_{r_{j}}\left( s_{m_{j-1}}a^{2(n_{j}-m_{j-1})}\right) ,
\end{equation*}%
which can be done since $\Phi _{r_{i}}(x)/\Phi _{r_{j}}(x)\rightarrow \infty 
$ as $x\rightarrow 0$.

We will assign ratio $a^{2}$ on levels $\{m_{j-1}+1,\ldots,n_{j}\}$, so $%
s_{m_{j-1}}a^{2(n_{j}-m_{j-1})}=s_{n_{j}}$ and that means (ii) actually says%
\begin{equation}
\Phi _{r_{i}}(s_{n_{j}})\geq \frac{8A}{B}\Phi _{r_{j}}\left(
s_{n_{j}}\right) \text{ whenever }i<j\text{ and }r_{i}>r_{j}.
\label{property2}
\end{equation}%
Choose the minimal integer $\ell _{j}$ such that $f(r_{j})^{\ell _{j}}\leq
S_{n_{j}}^{\Phi _{r_{j}}(n_{j})},$ put $m_{j}=4(n_{j}+\ell _{j})$ and assign
the ratios on levels $\{n_{j}+1,$ $\ldots,n_{j}+\ell _{j}\}$ to be $f(r_{j})$
and the ratios on the levels $\{n_{j}+\ell _{j}+1,..,m_{j}\}$ to be $a^{2}.$

Note that $\ell _{j}=\phi _{r_{j}}(n_{j})$ and property (i) in the
definition of $n_{j},$ together with (\ref{4}) and (\ref{5}), ensures $\ell
_{j}\leq n_{j}/8$. In particular, $n_{j}-\ell _{j}\geq 7/8n_{j}\geq 7m_{j-1}$
and $n_{j}+\ell _{j}=m_{j}/4$.

This completes the construction of $E$.

\smallskip \medskip

\textbf{Verification of the }$\Phi $\textbf{-dimensions:}

We now need to verify that we obtain the desired value for each $\overline{%
\dim }_{\Phi _{q}}E$. We can easily see that $\overline{\dim }_{\Phi
_{q}}E\geq \log 2/\left\vert \log f(q)\right\vert $ by noting that 
\begin{equation*}
\left( \frac{s_{n_{j}+\phi _{r_{j}}(n_{j})}}{s_{n_{j}}}\right) ^{1/\phi
_{r_{j}}(n_{j})}=f(r_{j})
\end{equation*}%
for the infinitely many choices of $r_{j}=q$. So we only need to prove the
other inequality.

Assume the first occurrence of $q$ in $\{r_{i}\}$ is with $i=j_{0}$. It will
be enough to show that $\left( s_{k+m}/s_{k}\right) ^{1/m}\leq f(q)$
whenever $k\geq n_{j_{0}}$ and $m\geq \phi _{q}(k)$. In other words, we want
to prove that the geometric mean of the ratios $r_{k+1},\ldots,r_{k+m}$ is
at most $f(q)$ for all $m\geq \phi _{q}(k)$ and $k\geq n_{j_{0}}$. A key
point to observe is that the geometric mean of any collection of ratios
where there are at least as many ratios equal to $a^{2}$ as otherwise, is at
most $a\leq f(q)$ for any $q$.

Given $k\geq N_{j_{0}},$ choose $j\geq j_{0}$ such that $k\in
\{m_{j-1}+1,\ldots,m_{j}\}:=B_{j}$. If either $k\leq n_{j}-\ell _{j}$ or $%
k>n_{j}+\ell _{j},$ then this is the situation with respect to the ratios $%
r_{k+1},\ldots,r_{k+m}$ (regardless of the size of $m$), so the geometric
mean is suitably small.

Thus we can assume $k\in \{n_{j}-\ell _{j}+1,n_{j}+\ell _{j}\}$. If $%
r_{j}\geq q=r_{j_{0}},$ then $f(r_{j})\leq f(q)$ and hence all ratios from $%
B_{j}$ are at most $f(q)$. In this case it is clear that the geometric mean
of the collection $r_{k+1},\ldots,r_{J},$ where $J=\min (k+m,m_{j}),$ is at
most $f(q)$. If $k+m>m_{j},$ then the set of ratios $\{r_{m_{j}+1},%
\ldots,r_{k+m}\}$ contains more ratios equal to $a^{2}$ than otherwise, so
its geometric mean is even at most $a$ and thus the geometric mean of the
full collection $\{r_{k+1},\ldots,r_{k+m}\}$ is at most $f(q)$.

The last case to consider is that for this choice of $j$ (which we remind
the reader is $\geq j_{0})$, we have $r_{j}<q=r_{j_{0}}$ and therefore $%
f(r_{j})>f(q)$. From (\ref{property2}), we note that 
\begin{equation*}
\Phi _{q}(s_{n_{j}})=\Phi _{r_{j_{0}}}(s_{n_{j}})\geq \frac{8A}{B}\Phi
_{r_{j}}(s_{n_{j}})\text{.}
\end{equation*}

The remaining arguments are now similar to the proof of Theorem \ref{Diff}.
Recall that $k\in \{n_{j}-\ell _{j}+1,n_{j}+\ell _{j}\}$. If $n_{j}-\ell
_{j}<k\leq n_{j}$, then $k\geq I_{r_{j}},$ so 
\begin{eqnarray}
\phi _{q}(k) &\geq &Bk\Phi _{q}(s_{k})\geq B(n_{j}-\ell _{j})\Phi
_{q}(s_{n_{j}})  \notag \\
&\geq &\frac{7}{8}Bn_{j}\Phi _{q}(s_{n_{j}})\geq 7An_{j}\Phi
_{r_{j}}(s_{n_{j}})\geq 7\phi _{r_{j}}(n_{j})=7\ell _{j}\text{.}  \label{6}
\end{eqnarray}%
The fact that $n_{j}\geq 8J_{q}$ also guarantees that $\phi _{q}(k)\leq
Ak\Phi _{q}(s_{k})\leq k/8,$ so $k+\phi _{q}(k)<m_{j}/2$. Thus the
collection $\{r_{k+1},\ldots,r_{k+\phi _{q}(k)}\}$ contains at most $\ell
_{j}$ terms of ratio $f(r_{j})$ and at least $6\ell _{j}$ terms of ratio $%
a^{2},$ and therefore has geometric mean at most $a$.

If, instead, $n_{j}<k\leq n_{j}+\ell _{j},$ then, as in the proof of Theorem %
\ref{Diff} (see particularly (\ref{Decrphi})), 
\begin{equation*}
k+\phi _{q}(k)-(n_{j}+\ell _{j})\geq \phi _{q}(n_{j})-\phi
_{r_{j}}(n_{j})\geq 6\ell _{j}
\end{equation*}%
where the final inequality comes from applying (\ref{6}) with $k=n_{j}$.
Again, 
\begin{equation*}
k+\phi _{q}(k)\leq 9k/8<m_{j}/2
\end{equation*}%
and thus again we deduce that the geometric mean of $\{r_{k+1},\ldots,r_{k+%
\phi _{q}(k)}\}$ is at most $a$.

For either choice of $k$, if $m>\phi _{q}(k),$ then since $n_{j}+\ell _{j}$ $%
<k+\phi _{q}(k)\leq m_{j}/2$ the collection of ratios $\{r_{k+\phi
_{q}(k)+1},\ldots,r_{k+m}\}$ has more that are value $a^{2}$ than otherwise,
and hence has geometric mean at most $a,$ as well. Thus we conclude $\left(
s_{k}/s_{k+m}\right) ^{1/m}\leq a$ in this (final) case.

This completes the proof.
\end{proof}

\begin{corollary}
\label{Cor:continuum}Given $0<\alpha <\beta <1,$ there is a set $E\subseteq 
\mathbb{[}0,1]$ such that 
\begin{equation*}
\{\overline{\dim }_{\Phi }E:\Phi \in \mathcal{D}\text{, }\lim_{x\rightarrow
0}\Phi (x)=0\}=[\alpha ,\beta ]=[\dim _{qA}E,\dim _{A}E].
\end{equation*}
\end{corollary}

\begin{proof}
Let $D(E)=\{\overline{\dim }_{\Phi }E:\Phi \in \mathcal{D}$, $%
\lim_{x\rightarrow 0}\Phi (x)=0\}$. We will let $\Phi _{p}(x)=|\log x|^{p-1}$
for $p\in (0,1)$. Clearly, 
\begin{equation*}
\{\overline{\dim }_{\Phi _{p}}E:\text{ }p\in (0,1)\}\subseteq D(E)\subseteq
\lbrack \dim _{qA}E,\dim _{A}E],
\end{equation*}%
so it will be sufficient to construct a set $E$ with $\{\overline{\dim }%
_{\Phi _{p}}E:p\in (0,1)\}=(\alpha ,\beta )$ and $\dim _{qA}E=\alpha ,$ $%
\dim _{A}E=\beta $. The previous theorem would permit us to construct such a
set satisfying the first property and would also have $\dim _{A}E=\beta $.
However, its quasi-Assouad dimension is $a^{2},$ so we need to modify the
construction slightly.

We can do this by requiring the sequence $\{n_{j}\}$ to grow so rapidly that
in addition to the requirements from before, we can also have $k_{j}$ much
greater than $m_{j}$ and $n_{j+1}$ much greater than $2k_{j}$. On the levels 
$k_{j}+1,\ldots,2k_{j}$ we will set the ratios to equal to $a=2^{-1/\alpha }$
(rather than $a^{2}$). One can see that $\dim _{qA}E=\log 2/|\log a|=\alpha $
by considering the terms $s_{k_{j}}/s_{2k_{j}}$. The sparseness of the $%
\{k_{j}\}$ will ensure that the other dimensions are not affected by this
change. We leave the technical details to the reader.
\end{proof}

\section{$\Phi $-dimensions of complementary sets in $\mathbb{R}$\label%
{Rearrangements}}

\label{sec:complementarysets}

\subsection{Bounds for $\Phi $-dimensions of complementary sets}

\label{subsec:complsets}

\subsubsection{Complementary sets}

Every closed subset of the interval $[0,1]$ of Lebesgue measure zero is of
the form $E=[0,1]\diagdown \bigcup_{j=1}^{\infty }U_{j}$ where $\{U_{j}\}$
is a disjoint family of open subintervals of $[0,1]$ whose lengths sum to
one. We will let $a=\{a_{j}\}_{j=1}^{\infty }$ where $a_{j}$ is the length
of $U_{j}$. Of course, $\sum_{j}a_{j}=1$ and without loss of generality we
can assume $a_{j+1}\leq a_{j}$. We will denote by $\mathcal{C}_{a}$ the
collection of all such closed sets $E$. These are called the \textbf{%
complementary sets of }$a$.

One example of a complementary set is the Cantor set associated with $a,$
denoted $C_{a}$. Another is the countable set, $D_{a},$ called the
decreasing rearrangement, defined as%
\begin{equation*}
D_{a}=\left\{ \sum_{i\geq k}a_{i}:k=1,2,\ldots\right\}
=\{1,1-a_{1},1-a_{1}-a_{2},\ldots\}.
\end{equation*}

As is well known, all complementary sets of a given sequence $a$ have the
same upper and lower box dimensions \cite[Section 3.2]{Fal}, but, of course,
this need not be true for other dimensions. For instance, the Hausdorff
dimension of the decreasing rearrangement is $0$, but this need not be true
for the Cantor set. In \cite{BT}, Besicovitch and Taylor proved that the
Cantor set $C_{a}$ had the maximum Hausdorff dimension of any set in $%
\mathcal{C}_{a}$. Further, they showed given any $s\leq \dim _{H}C_{a}$
there is some set $E\in \mathcal{C}_{a}$ with $\dim _{H}E=s$. The same
result was shown to be true with the Hausdorff dimension replaced by the
packing dimension in \cite{HMZ}. In \cite{GHM}, it was shown that the Cantor
set and the decreasing set also have the extremal Assouad dimensions (under
natural assumptions on the gap sequence $a$). But unlike the situation for
Hausdorff, packing and lower Assouad dimensions, $\dim _{A}C_{a}$ is minimal
among the sets in $\mathcal{C}_{a}$ and $\dim _{A}D_{a}$ is maximal (and
equals $1$ for such $a$). Again, it was shown that the full range of
possible dimensions is attained, namely $\{\dim _{L}E:E\in \mathcal{C}%
_{a}\}= $ \thinspace $\lbrack 0,\dim _{L}C]$ and $\{\dim _{A}E:E\in \mathcal{%
C}_{a}\}=[\dim _{A}C,1]$.

In this section, we will prove analogous results for the $\Phi $-dimensions,
although some proofs are necessarily quite different.

\subsubsection{Decreasing rearrangement}

We first prove that the decreasing rearrangement is always one of the
extreme values of the $\Phi $-dimension over the class $\mathcal{C}_{a}$.
This requires a proof in the case of the upper $\Phi $-dimension as this
dimension need not be one, c.f., Example \ref{Ex:Dec}. To begin, we first
point out the following elementary result which essentially can be found in 
\cite{Fal}.

\begin{lemma}
\label{box}Suppose $F,G$ are two compact sets in $\mathcal{C}_{a}$ for some
decreasing, summable sequence $a=(a_{n})$. For any $r>0$, 
\begin{equation*}
\frac{1}{16}\leq \frac{N_{r}(F)}{N_{r}(G)}\leq 16.
\end{equation*}
\end{lemma}

\begin{proposition}
\label{Dec}If $a$ is any decreasing, summable sequence, then $\overline{\dim 
}_{\Phi }E\leq \overline{\dim }_{\Phi }D_{a}$ and \underline{$\dim $}$_{\Phi
}E\geq $ $\underline{\dim }_{\Phi }D_{a}=0$ for all $E\in \mathcal{C}_{a}$.
\end{proposition}

\begin{proof}
As $D_{a}$ has isolated points, $\underline{\dim }_{\Phi }D_{a}=0$ for all
dimensions functions $\Phi $ and hence is the minimal lower $\Phi $%
-dimension.

To prove that $\overline{\dim }_{\Phi }D_{a}$ is the maximal upper $\Phi $%
-dimension we will simply show that 
\begin{equation*}
N_{r}(E\cap B(z,R))\leq 64N_{r}(D_{a}\cap B(0,R))
\end{equation*}%
for all $z\in E$ and $r\leq R$.

To see this, let $\{a_{j_{i}}\}$ be the set of gap lengths that are
completely contained in $E\cap B(z,R)$ and let $R^{\prime }=\sum a_{j_{i}} $%
. Then $N_{r}(E\cap B(z,R))=N_{r}(E\cap I)$ for a suitable interval $I$ $%
\subseteq B(z,R)$ of length $R^{\prime }\leq 2R$. Let $E^{\prime }$ denote
the set formed by removing from $[0,R^{\prime }]$ the gaps of lengths $%
\{a_{j_{i}}\}$ in decreasing order (from right to left). By Lemma \ref{box}, 
$N_{r}(E\cap I)$ $\leq 16N_{r}(E^{\prime }\cap [0,R^{\prime }])$.

Choose $n$ such that $a_{n}\leq 2r<a_{n-1}$ and suppose that 
\begin{equation*}
\sum_{i=m+1}^{\infty }a_{i}<R^{\prime }\leq \sum_{i=m}^{\infty }a_{i}
\end{equation*}%
(we will say that $R^{\prime }$ belongs to gap $a_{m}$ in the set $D_{a}$).
If $n\leq m,$ then all gaps in the construction of $D_{a}$ intersecting $%
[0,R^{\prime }]$ have length at most $2r$. Thus $N_{r}(D_{a}\cap
[0,R^{\prime }])=\lceil \frac{R^{\prime }}{2r}\rceil $ is the maximum
possible value and hence it dominates $N_{r}(E^{\prime }\cap [0,R^{\prime
}]) $.

So assume $n>m$ and let $A=\sum_{i=n}^{\infty }a_{i}<R^{\prime }$. As $%
a_{j}\leq 2r$ for $j\geq n$ and $a_{j}>2r$ for $j\leq n-1$, 
\begin{eqnarray*}
N_{r}(D_{a}\cap [0,R^{\prime }]) &\geq &N_{r}(D_{a}\cap
[0,A])+N_{r}(D_{a}\cap [A+a_{n-1},R^{\prime }]) \\
&\geq &\left\lceil \frac{A}{2r}\right\rceil +\max (n-m-2,0)
\end{eqnarray*}%
(where $[A+a_{n-1},R^{\prime }]$ is empty if $R^{\prime }<A+a_{n-1}$).

Assume that the number $A$ belongs to gap $a_{j_{s}}$ in the set $E^{\prime
},$ so%
\begin{equation*}
N_{r}(E^{\prime }\cap [0,R^{\prime }])\leq \left\lceil \frac{A}{2r}%
\right\rceil +s.
\end{equation*}%
Since $\sum_{i=n}^{\infty }a_{i}=A\geq \sum_{i=s+1}^{\infty }a_{j_{i}}$, it
follows that $j_{s}\leq n-1$ and consequently, $a_{j_{s-k}}\geq a_{n-1-k}$
for all $k=0,\ldots,s-1$. Thus if $n-s+1\leq m,$ then%
\begin{eqnarray*}
R^{\prime }-A &>&a_{j_{s-1}}+a_{j_{s-2}}+\cdot \cdot \cdot
+a_{j_{2}}+a_{j_{1}}\geq a_{n-2}+a_{n-3}+\cdot \cdot \cdot +a_{n-s+1}+a_{n-s}
\\
&\geq &a_{n-2}+a_{n-3}+\cdot \cdot \cdot +a_{n-s+1}+a_{n-1}\geq R^{\prime
}-A.
\end{eqnarray*}%
This contradiction proves $s\leq n-m$. Thus 
\begin{equation*}
N_{r}(D_{a}\cap [0,R^{\prime }])\geq \left\lceil \frac{A}{2r}\right\rceil
+\max (s-2,0),
\end{equation*}%
from which it is easy to check that 
\begin{eqnarray*}
N_{r}(D_{a}\cap [0,R]) &\geq &N_{r}(D_{a}\cap [0,R^{\prime }])\geq \frac{1}{4%
}\left( \left\lceil \frac{A}{2r}\right\rceil +s\right) \\
&\geq &\frac{1}{4}N_{r}(E^{\prime }\cap [0,R^{\prime }])\geq \frac{1}{64}%
N_{r}(E\cap B(z,R)),
\end{eqnarray*}%
and the proposition follows.
\end{proof}

\subsubsection{Cantor Sets}

We now focus our attention on decreasing, summable sequences $a$ with the
property that there are constants $\tau $ and $\lambda $ with%
\begin{equation}
0<\tau \leq s_{j+1}/s_{j}\leq \lambda <1/2.  \label{doubling}
\end{equation}%
Here, as before, $s_{j}=2^{-j}\sum_{i\geq 2^{j}}a_{i}$. We will call a
sequence $\{a_{j}\}$ with this property \textbf{level comparable.} Of
course, the doubling assumption automatically gives the left hand inequality
and central Cantor sets have the level comparable property precisely when
their ratios of dissection are bounded away from $0$ and $1/2$.

The level comparable assumption is very useful as it ensures that $s_{k}\sim
a_{2^{k}}$ since $s_{k}\geq a_{2^{k+1}}\gtrsim a_{2^{k}}$ and 
\begin{equation}
(1-2\lambda )s_{k}\leq s_{k}-2s_{k+1}\leq a_{2^{k}}\text{ .}  \label{ineq3}
\end{equation}%
We remind the reader that the symbols $\sim $ and $\gtrsim $ were defined in
Notation \ref{sim}.

For level comparable sequences, the Cantor set has the other extreme value
for the $\Phi $-dimensions.

\begin{theorem}
\label{Cantorbd}If $a=\{a_{j}\}$ is a level comparable sequence and $\Phi$ is a dimension function, then for
all $E\in \mathcal{C}_{a}$ we have $\overline{\dim }_{\Phi }E\geq \overline{%
\dim }_{\Phi }C_{a}$ and $\underline{\dim }_{\Phi }E\leq \underline{\dim }%
_{\Phi }C_{a}$.
\end{theorem}

\begin{proof}
We begin with the upper $\Phi $-dimension. Observe that if $\phi $ is
bounded, then the upper $\Phi $-dimension is the Assouad dimension and the
result is already known in that case, see \cite[Thm. 3.5]{GHM}. So assume
otherwise. Some modifications to the proof of Theorem 3.5 in \cite{GHM} are
required.

Let $d=\overline{\dim }_{\Phi }C_{a}$. From the formula for the upper $\Phi $%
-dimension of $C_{a},$ Theorem \ref{Cantorformula}, we know there must exist 
$\kappa _{0},c_{0}$ and indices $k\geq \kappa _{0}$ and $n\geq \phi (k)$
such that 
\begin{equation*}
c_{0}2^{n}\geq \left( \frac{s_{k}}{s_{k+n}}\right) ^{d-\varepsilon }\geq
2^{n\varepsilon }\left( \frac{s_{k}}{s_{k+n}}\right) ^{d-2\varepsilon },
\end{equation*}%
where the latter inequality holds because $s_{k}/s_{k+1}\geq 2$ for all $k$.

We will refer to the complementary gaps of lengths $a_{2^{k-1}},%
\ldots,a_{2^{k-1}}$ as the gaps of level $k$.

Remove from $[0,\sum a_{j}]$ the complementary gaps of levels $1,\ldots,k$
to obtain the set $J_{1}\cup \cdots \cup J_{M_{k}}\cup \{$singletons\} where 
$J_{i}$ are non-trivial, closed, disjoint intervals, $M_{k}\leq 2^{k}$ and $%
\sum_{i}\left\vert J_{i}\right\vert =2^{k}s_{k}$. Let $b_{i}$ denote the
number of gaps of step $k+n$ contained in $J_{i}$ and put $r=a_{2^{k+n}}/2$.
If we let $x_{i}$ be an endpoint of $J_{i}$, then as the gaps of step $k+n$
are at least $2r$ in length, $N_{r}(B(x_{i},\left\vert J_{i}\right\vert
)\cap E)\geq b_{i}$. Since $\sum_{i}b_{i}=2^{k+n},$ 
\begin{equation*}
\sum_{i}N_{r}(B(x_{i},\left\vert J_{i}\right\vert )\cap E)\geq 2^{k+n}\geq 
\frac{2^{k}2^{n\varepsilon }}{c_{0}}\left( \frac{s_{k}}{s_{k+n}}\right)
^{d-2\varepsilon }.
\end{equation*}

Let 
\begin{equation*}
\mathcal{I}=\{i\in \{1,\ldots,M_{k}\}:\left\vert J_{i}\right\vert \leq
s_{k}\}.
\end{equation*}%
If there is some index $i\in I$ with $N_{r}(B(x_{i},\left\vert
J_{i}\right\vert )\cap E)\geq \left( s_{k}/s_{k+n}\right) ^{d-2\varepsilon }$%
, then 
\begin{equation}
N_{r}(B(x_{i},s_{k})\cap E)\geq N_{r}(B(x_{i},\left\vert J_{i}\right\vert
)\cap E)\geq \left( \frac{s_{k}}{s_{k+n}}\right) ^{d-2\varepsilon }.
\label{Ineq1}
\end{equation}%
Otherwise, 
\begin{eqnarray*}
\sum_{i\notin \mathcal{I}}N_{r}(B(x_{i},\left\vert J_{i}\right\vert )\cap E)
&=&\sum_{i}N_{r}(B(x_{i},\left\vert J_{i}\right\vert )\cap E)-\sum_{i\in 
\mathcal{I}}N_{r}(B(x_{i},\left\vert J_{i}\right\vert )\cap E) \\
&\geq &2^{k+n}-\left\vert \mathcal{I}\right\vert \max_{i\in \mathcal{I}%
}N_{r}(B(x_{i},\left\vert J_{i}\right\vert )\cap E) \\
&\geq &\frac{2^{k}2^{n\varepsilon }}{c_{0}}\left( \frac{s_{k}}{s_{k+n}}%
\right) ^{d-2\varepsilon }-2^{k}\left( \frac{s_{k}}{s_{k+n}}\right)
^{d-2\varepsilon } \\
&\geq &2^{k}\left( \frac{s_{k}}{s_{k+n}}\right) ^{d-2\varepsilon }\left( 
\frac{2^{n\varepsilon }}{c_{0}}-1\right) .
\end{eqnarray*}%
Since $n\geq \phi (k)\rightarrow \infty ,$ we can assume $2^{n\varepsilon
}/c_{0}-1\gtrsim 2^{n\varepsilon }$. Recall that $\sum_{i}\left\vert
J_{i}\right\vert =2^{k}s_{k},$ thus 
\begin{equation*}
\sum_{i\notin \mathcal{I}}N_{r}(B(x_{i},\left\vert J_{i}\right\vert )\cap
E)\gtrsim 2^{k}\left( \frac{s_{k}}{s_{k+n}}\right) ^{d-2\varepsilon
}2^{n\varepsilon }\gtrsim 2^{k}2^{n\varepsilon }\left( \frac{%
2^{-k}\sum_{i\notin \mathcal{I}}\left\vert J_{i}\right\vert }{a_{2^{k+n}}}%
\right) ^{d-2\varepsilon }
\end{equation*}%
An application of Holder's inequality gives%
\begin{eqnarray*}
\sum_{i\notin \mathcal{I}}N_{r}(B(x_{i},\left\vert J_{i}\right\vert )\cap E)
&\gtrsim &2^{k}2^{n\varepsilon }\left( \frac{2^{-k}}{a_{2^{k+n}}}\right)
^{d-2\varepsilon }\sum_{i\notin \mathcal{I}}\left\vert J_{i}\right\vert
^{d-2\varepsilon }\left\vert \mathcal{I}^{c}\right\vert ^{-(1-d+2\varepsilon
)} \\
&\gtrsim &2^{n\varepsilon }\left( \frac{2^{k}}{\left\vert \mathcal{I}%
^{c}\right\vert }\right) ^{1-d+2\varepsilon }\frac{\sum_{i\notin \mathcal{I}%
}\left\vert J_{i}\right\vert ^{d-2\varepsilon }}{r^{d-2\varepsilon }}\gtrsim
2^{n\varepsilon }\frac{\sum_{i\notin \mathcal{I}}\left\vert J_{i}\right\vert
^{d-2\varepsilon }}{r^{d-2\varepsilon }},
\end{eqnarray*}%
with the final inequality arising because $\left\vert \mathcal{I}%
^{c}\right\vert \leq M_{k}\leq 2^{k}$ and $d\leq 1$. It follows that in this
case there must be some choice of $i\notin \mathcal{I}$ such that 
\begin{equation}
N_{r}\left( B(x_{i},\left\vert J_{i}\right\vert )\cap E\right) \geq
c2^{n\varepsilon }\left( \frac{\left\vert J_{i}\right\vert }{r}\right)
^{d-2\varepsilon }.  \label{Ineq2}
\end{equation}%
By definition, $i\notin \mathcal{I}$ implies $\left\vert J_{i}\right\vert
\geq s_{k}$ and thus $\left\vert J_{i}\right\vert ^{1+\Phi (|J_{i}|)}\geq
s_{k}^{1+\Phi (s_{k})}\geq s_{k+n}\sim r$.

As either (\ref{Ineq1}) or (\ref{Ineq2}) must hold, we deduce that $%
\overline{\dim }_{\Phi }E\geq d-2\varepsilon $ and that gives the desired
result.

The proof for the lower $\Phi $-dimension is a straightforward modification
of Theorem 4.1 of \cite{GHM}.
\end{proof}

Combining Proposition \ref{Dec} and Theorem \ref{Cantorbd} gives the
following statement.

\begin{corollary}
\label{Cor:IntDim}If $a$ is any level comparable sequence, then for all $%
E\in \mathcal{C}_{a}$ we have $\overline{\dim }_{\Phi }E\in \left[ \overline{%
\dim }_{\Phi }C_{a},\overline{\dim }_{\Phi }D_{a}\right] $ and \underline{$%
\dim $}$_{\Phi }E\in \left[ 0,\underline{\dim }_{\Phi }C_{a}\right] $. In
particular, these statements are true for the quasi-Assouad dimensions.
\end{corollary}

\subsection{An interval of $\Phi $-dimensions for complementary sets}

In \cite{GHM} it was shown that if $a$ is any level comparable sequence,
then for every $c\in \lbrack 0,\dim _{L}C_{a}]$ and $d\in $ $[\dim
_{A}C_{a},1]$ there are sets $E_{c},E_{d}\in \mathcal{C}_{a}$ with $\dim
_{L}E_{c}=c$ and $\dim _{A}E_{d}=d$.\footnote{%
Actually, the assumption that $a$ is doubling suffices for the upper Assouad
dimension.} These results continue to be true for the quasi-Assouad and $%
\Phi $-dimensions when $\Phi \rightarrow \delta$, with $\delta\in \lbrack
0,\infty ]$. For the lower $\Phi $-dimensions essentially the same proof as
given in \cite{GHM} for the lower Assouad dimension works. We give a brief
sketch of the main idea at the beginning of the proof of Theorem \ref{Thm5}.

For the upper $\Phi $-dimension, note that the case $\delta=\infty $ is
trivial since we recover the upper box dimension and all complementary sets
of a given sequence have the same upper box dimension. Different proofs are
required for the cases $\Phi \rightarrow \delta$ for $\delta=0$ or $\delta>0$%
, and these are necessarily different from the proof given for the Assouad
dimension in \cite{GHM} as the set constructed there only exhibits large
local `thickness' on scales $r$ that are nearly as large as $R,$ and hence
are not suitable for use in obtaining these other dimensions.

\begin{theorem}
\label{Thm5} Suppose $a$ is a level comparable sequence and $\Phi $ is a
dimension function with $\Phi (x)\rightarrow \delta$, for some $\delta\in[%
0,\infty]$. Then for every $c\in \lbrack 0,\underline{\dim }_{\Phi }C_{a}]$
and $d\in \lbrack \overline{\dim }_{\Phi }C_{a},\overline{\dim}_\Phi D_a],$
there are sets $E_{c},E_{d}\in \mathcal{C}_{a}$ with $\underline{\dim }%
_{\Phi }E_{c}=c$ and $\overline{\dim }_{\Phi }E_{d}=d$. A similar statement
holds with $\underline{\dim }_{\Phi }C_{a}$ replaced by $\dim _{qL}C_{a}$
and $\overline{\dim }_{\Phi }C_{a}$ replaced by $\dim _{qA}C_{a}$.
\end{theorem}

\begin{remark}
We remind the reader that for any doubling sequence $a$ (and hence any level
comparable sequence) and any dimension function $\Phi \rightarrow 0$, we
have $\overline{\dim }_{B}D_{a}>0$ and thus $\overline{\dim }_{\Phi
}D_{a}\geq \dim _{qA}D_{a}=1$ by \cite{GH}.
\end{remark}

Combining this result with Theorem \ref{Cantorbd} gives the following.

\begin{corollary}
Suppose $a$ is any level comparable sequence and $\Phi $ is a dimension
function with $\Phi (x)\rightarrow \delta$. Then%
\begin{equation*}
\{\overline{\dim }_{\Phi }E:E\in \mathcal{C}_{a}\}=\left[ \overline{\dim }%
_{\Phi }C_{a},\overline{\dim }_{\Phi }D_{a}\right]
\end{equation*}%
and%
\begin{equation*}
\{\underline{\dim }_{\Phi }E:E\in \mathcal{C}_{a}\}=\left[ \underline{\dim }%
_{\Phi }D_{a},\underline{\dim }_{\Phi }C_{a}\right] =\left[ 0,\underline{%
\dim }_{\Phi }C_{a}\right] .
\end{equation*}
\end{corollary}

\begin{proof}[Proof of Theorem \protect\ref{Thm5}]
For the lower dimension case, the same proof given in \cite[Theorem 4.3]{GHM}
for the lower Assouad dimension, with the obvious modifications, works for
the lower $\Phi $-dimensions and the lower quasi-Assouad dimension. A sketch
of the proof is that for $0<\alpha <\underline{\dim }_{\Phi }C_{a}$, it is
possible to find a subsequence of $a$ whose Cantor rearrangement is an $%
\alpha $-Ahlfors regular set and such that the Cantor rearrangement of the
remaining gaps has lower $\Phi $-dimension equal to $\underline{\dim }_{\Phi
}C_{a}$. This gives a complementary set $E$ with $\underline{\dim }_{\Phi
}E=\alpha $.

For the upper $\Phi $-dimension problem, we first remark that if $\Phi
\rightarrow \infty $, then $\overline{\dim }_{\Phi }E=\overline{\dim }_{B}E$
and all sets $E\in \mathcal{C}_{a}$ have the same upper box dimension.

Thus it remains to study the upper $\Phi $-dimension problem when $\Phi
\rightarrow \delta \in \lbrack 0,\infty )$. We will first give the proof for
the case $\Phi \rightarrow \delta \neq 0$, where we can take advantage of an
explicit formula for the $\Phi $-dimension of the decreasing rearrangement.
The harder case, $\delta =0$, will be done second. \vspace{0.2cm}

\noindent \textbf{Case $\mathbf{\Phi \rightarrow \delta }\in (0,\infty )$.}

According to Corollary \ref{Coro:p}, $\overline{\dim }_{\Phi }E=\overline{%
\dim }_{A}^{\theta }E$ for all $E$, where $\theta =(1+\delta )^{-1}$. 
Observe that for any decreasing set $D$, 
\begin{equation}
\overline{\dim }_{A}^{\theta }D=\min \left( \frac{\overline{\dim }_{B}D}{%
1-\theta },1\right) .  \label{formula}
\end{equation}%
This follows from \cite[Theorem 6.2]{FYAdv} and \cite[Theorem 2.1]{FTale}.

Given any $0<d<\overline{\dim }_{A}^{\theta }D_{a}$, we will use the above
formula to construct a subsequence $b$ of $a$ such that if $\widetilde{a}$
is the subsequence obtained after removing $b$ from $a$, then $\overline{%
\dim }_{A}^{\theta }D_{b}=d$ and $\overline{\dim }_{A}^{\theta }C_{%
\widetilde{a}}=\overline{\dim }_{A}^{\theta }C_{a}$. The set $%
E_{d}=D_{b}\cup C_{\widetilde{a}}$ will belong to $\mathcal{C}_{a}$ and by
the union property, its upper $\theta $-dimension will be given by 
\begin{equation*}
\overline{\dim }_{A}^{\theta }E_{d}=\max (d,\overline{\dim }_{A}^{\theta
}C_{a}),
\end{equation*}%
which will prove the statement of the theorem.

Let $\sigma :=\overline{\dim }_{B}D_{a}$. By \cite[Section 3.4]{Tricot} we
have 
\begin{equation*}
\sigma =\limsup_{n\rightarrow \infty }\frac{\log n}{-\log a_{n}}%
=\lim_{k\rightarrow \infty }\frac{\log n_{k}}{-\log a_{n_{k}}},
\end{equation*}%
where $\{n_{k}\}$ is chosen to be a suitable sequence, say $n_{k+1}\geq
2^{n_{k}}$. Choose $\gamma $ such that $\gamma \sigma =d(1-\theta )$ and
define the subsequence $b$ by 
\begin{equation*}
b_{m}=\left\{ 
\begin{array}{cc}
a_{n_{k}}, & \lfloor m^{1/\gamma }\rfloor \leq n_{k}<\lfloor (m+1)^{1/\gamma
}\rfloor \\ 
a_{\lfloor m^{1/\gamma }\rfloor }, & {\text{ otherwise}}%
\end{array}%
\right. .
\end{equation*}

Note that for the integers $m_{k}$ where $b_{m_{k}}=a_{n_{k}}$ we have $%
m_{k}\sim n_{k}^{\gamma }$, so 
\begin{equation*}
\lim_{k}\frac{\log m_{k}}{-\log b_{m_{k}}}=\lim_{k}\frac{\log n_{k}^{\gamma }%
}{-\log a_{n_{k}}}=\sigma \gamma .
\end{equation*}%
Moreover, for $\epsilon >0$ we have $\log n/(-\log a_{n})<\sigma +\epsilon $
for all $n$ large enough, so for large $m$, with $m\neq m_{k}$, 
\begin{equation*}
\frac{\log m}{-\log b_{m}}=\frac{\gamma \log m^{1/\gamma }}{-\log a_{\lfloor
m^{1/\gamma }\rfloor }}\leq \gamma (\sigma +2\epsilon ).
\end{equation*}%
Therefore, $\overline{\dim }_{B}D_{b}=\sigma \gamma $, and by (\ref{formula}%
) $\overline{\dim }_{A}^{\theta }D_{b}=d$.

Finally, note that $\lfloor (m+1)^{1/\gamma }\rfloor -\lfloor m^{1/\gamma
}\rfloor \rightarrow \infty $ as $m$ increases. As the original sequence was
doubling, this ensures that the sequence $\widetilde{a}$ consisting of the
remaining gaps is comparable to the original sequence $a$. In consequence, $%
\overline{\dim }_{A}^{\theta }C_{\widetilde{a}}=\overline{\dim }_{A}^{\theta
}C_{a}$ and, as we noted above, that completes the proof in this case.

\smallskip

\noindent \textbf{Case $\mathbf{\Phi \rightarrow 0}$.}

We will give a detailed proof for the quasi-Assouad dimension. It will be
clear that the same arguments will work for the upper $\Phi $-dimension with 
$\Phi \rightarrow 0$. Our proof is constructive. The set $E=E_{d}\in 
\mathcal{C}_{a}$ will again have the form $E=\mathcal{A}\cup \mathcal{B}$,
with $\dim _{qA}\mathcal{A}$ equal to the desired in-between value $d$ and $%
\dim _{qA}\mathcal{B=}\dim _{qA}C_{a}$. The union property for the
quasi-Assouad dimension will ensure that $E$ has the desired quasi-Assouad
dimension.

If $b=\{b_{j}\}$ is the sequence with $b_{2^{j}+t}=a_{2^{j}}$ for $%
t=0,\ldots,2^{j}-1$, then $a,b$ are comparable sequences and if $E$ is the
set formed with some rearrangement of $a$ and $F$ is the corresponding
rearrangement of $b,$ then $E$ and $F$ are bi-Lipschitz equivalent. So
without loss of generality we will assume $a$ is constant along diadic
blocks. Moreover, a level comparable sequence $\{a_{j}\}$ has the property
that there are constants $u,v$ such that 
\begin{equation*}
1>u\geq \frac{a_{2^{j}}}{a_{2^{j-1}}}\geq v>0\text{ for all }j.
\end{equation*}

If $\dim _{qA}C_{a}=1$, there is nothing to do. So assume $1>d\geq \dim
_{qA}C_{a},$ say $d=\log 2/\left\vert \log \beta \right\vert $ where $\beta
<1/2$.

To simplify the notation, we will let $\alpha _{j}=a_{2^{j}}$. Temporarily
fix $m$. Given $j\geq 1,$ choose the minimal index $i(j)\geq 1$ such that $%
\alpha _{m+i(j)}/\alpha _{m}\leq \beta ^{j}$ and choose the maximal integer $%
J_{j}\geq 1$ such that 
\begin{equation*}
J_{j}\frac{\alpha _{m+i(j)}}{\alpha _{m}}\leq \beta ^{j}.
\end{equation*}%
The minimality of $i(j)$ ensures that 
\begin{equation*}
J_{j}\frac{\alpha _{m+i(j)}}{\alpha _{m}}\leq \beta ^{j}<\frac{\alpha
_{m+i(j)-1}}{\alpha _{m}}
\end{equation*}%
which implies 
\begin{equation*}
J_{j}\leq \frac{\alpha _{m+i(j)-1}}{\alpha _{m+i(j)}}\leq \frac{1}{v}.
\end{equation*}%
Similarly, the maximality of $J_{j}$ means that%
\begin{equation*}
(J_{j}+1)\frac{\alpha _{m+i(j)}}{\alpha _{m}}>\beta ^{j},
\end{equation*}%
so 
\begin{equation*}
\alpha _{m+i(j)}\geq \frac{\alpha _{m}\beta ^{j}}{1+1/v}=c_{1}\alpha
_{m}\beta ^{j}
\end{equation*}%
where $c_{1}>0$ is independent of $m$ and $j$. Moreover, the fact that $%
v^{i}\leq \alpha _{m+i}/\alpha _{m}\leq u^{i}$, coupled with the definition
of $i(j),$ implies%
\begin{equation*}
jc_{3}\leq j\frac{\log \beta }{\log v}+1\leq i(j)\leq j\frac{\log \beta }{%
\log v}=jc_{2}
\end{equation*}%
where we again note that $c_{2},c_{3}$ are positive constants, independent
of $m$ and $j$.

\medskip

\textbf{Construction of the set }$E$\textbf{:}

We now form a Cantor-tree like arrangement with blocks of gaps. The first
block will consist of $J_{1}$ gaps of length $\alpha _{m+i(1)}$ placed
adjacently. The blocks of level 2 will each consist of $J_{2}$ gaps of
length $\alpha _{m+i(2)}$ placed adjacently and there will be two blocks of
level 2, one to the left and the other to the right of the block of level 1.
In general, there will be $2^{j-1}$ blocks of level $j,$ each consisting of $%
J_{j}$ gaps of length $\alpha _{m+i(j)}$ placed in a Cantor-like
arrangement. If we do this for $j=1,\ldots,n$, we will call the resulting
finite set $X_{m,n}$. Note that the length of any block of level $j$ in $%
X_{m,n}$ is equal to $J_{j}\alpha _{m+i(j)}$ and satisfies 
\begin{equation}
c_{1}\alpha _{m}\beta ^{j}\leq J_{j}\alpha _{m+i(j)}\leq \alpha _{m}\beta
^{j}.  \label{B1}
\end{equation}%
Hence the diameter of $X_{m,n}$ is at least the length of block 1 which is $%
\geq c_{1}\alpha _{m}\beta ,$ and the diameter of $X_{m,n}$ is at most 
\begin{equation}
\sum_{j=1}^{n}2^{j-1}\alpha _{m}\beta ^{j}\leq \alpha _{m}\frac{\beta }{%
1-2\beta }=c_{4}\alpha _{m}\beta .  \label{B2}
\end{equation}

Since $i(j)\geq c_{2}j$, for each $k$ the number of gaps of length $\alpha
_{m+k}$ that we will require is%
\begin{equation*}
\sum_{\substack{ j\in \{1,\ldots,n\}  \\ i(j)=k}}J_{j}2^{j-1}\leq
\sum_{j=1}^{k/c_{2}}J_{j}2^{j-1}\leq \frac{1}{v}2^{k/c_{2}}.
\end{equation*}%
As $j\in \{1,\ldots,n\}$ and $i(j)$ $\leq c_{3}j,$ we have $k\leq c_{3}n$.\ Of
course, for each $k$ there are a total of $2^{m+k}$ gaps of this size
available in the sequence $a$, so we have enough gaps, even twice as many as
we need, provided 
\begin{equation*}
\frac{1}{v}2^{k/c_{2}}\leq 2^{m+k-1}\text{ for each }k=1,\ldots,c_{3}n\text{.%
}
\end{equation*}%
Hence there is some $c_{5}>0$ (and independent of $m)$ such that if $n\leq
c_{5}m$, then there will be enough gaps to carry out this construction.

Lastly, we will select a rapidly growing sequence of integers $\{m_{k}\}$
and let $n_{k}=[c_{5}m_{k}]$. We will set $A_{k}=X_{m_{k},n_{k}}$. We will
want $m_{k+1}$ to be much larger than $m_{k}+c_{3}n_{k},$ so that we will
not use any gaps from the same diadic blocks in two different sets $A_{j}$.
Also, we will want to choose $m_{k}$ increasing so rapidly that the diameter
of $A_{k+1}$ is at most $1/2$ diameter of $A_{k}$.

We will position the sets $A_{k}$ adjacent to each other in decreasing order
and let%
\begin{equation*}
\mathcal{A}=\bigcup_{k=1}^{\infty }A_{k}.
\end{equation*}%
The gaps of the sequence $\{a_{j}\}$ that were not used in the construction
of the sets $A_{k}$ will be then placed to form a Cantor set $\mathcal{B}$
to the left of $A_{1}$. This completes the construction of the set $E=%
\mathcal{A}\cup \mathcal{B}\in C_{a}$.

\medskip

\textbf{Computation of }$\dim _{qA}E$\textbf{:}

Since there are at least half the gaps $a_{j}$ left in each diadic block,
the decreasing sequence consisting of the remaining gaps is comparable to
the original sequence. Hence $\dim _{qA}\mathcal{B}=\dim _{qA}C_{a}\leq d$.
Thus, to see that the rearranged set $\mathcal{A}\cup \mathcal{B}$ has
quasi-Assouad dimension $d,$ it will be enough to prove $\dim _{qA}\mathcal{A%
}=d$.

(a) Lower bound for $\dim _{qA}\mathcal{A}$:

We will let $\left\vert Y\right\vert $ denote the diameter of a set $%
Y\subseteq \mathbb{R}$.

To see that $\dim _{qA}\mathcal{A}\geq d$, consider $R=$ $\left\vert
A_{k}\right\vert $ $\sim \alpha _{m_{k}}\beta $ (by (\ref{B2})) and $r=\frac{%
1}{2}$length of block of level $n_{k}$ in $A_{k}$, so that $r\sim \alpha
_{m_{k}}\beta ^{n_{k}}$ (by (\ref{B1})). Notice that if $\delta >0$
(independent of $k$) is chosen such that $\beta ^{c_{5}}v^{-\delta }<1,$
then as $a_{2^{n}}\geq v^{n},$ for sufficiently large $k,$%
\begin{equation*}
\frac{r}{R^{1+\delta }}\leq c\frac{\alpha _{m_{k}}\beta ^{n_{k}}}{\left(
\alpha _{m_{k}}\beta \right) ^{1+\delta }}=c\frac{\beta ^{n_{k}}}{\beta
^{1+\delta }\alpha _{m_{k}}^\delta}\leq \frac{c\beta^{-1}}{\beta ^{1+\delta }}\left( \frac{%
\beta ^{c_{5}}}{v^{\delta }}\right) ^{m_{k}}<1.
\end{equation*}%
If we let $z\in A_{k},$ then $N_{r}(B(z,R)\cap A_{k})\geq 2^{n_{k}-1}$ since
the blocks of level $n_{k}$ are separated by at least $r$, while $%
(R/r)^{d}\sim \beta ^{-dn_{k}}=2^{n_{k}}$. In order for there to be a
constant $C$ such that $N_{r}(B(z,R)\cap A_{k})\leq C(R/r)^{t}$ for all $k$,
we must have $t\geq \log 2/\left\vert \log \beta \right\vert =d$. This shows 
$\dim _{qA}\bigcup A_{k}\geq d.$

\medskip

(b) Upper bound for $\dim _{qA}\mathcal{A}$ :

For this, we will prove the following claim.

\begin{claim}
There is a constant $C,$ independent of $k,$ such that%
\begin{equation}
N_{r}(B(z,R)\cap A_{k})\leq C\left( \frac{\min (\left\vert A_{k}\right\vert
,R)}{r}\right) ^{d}  \label{Adim1}
\end{equation}%
for all $r<\min (\left\vert A_{k}\right\vert ,R)$ and all $z\in \mathcal{A}$.
\end{claim}

Assuming the claim, we can even prove that $\dim _{A}\mathcal{A}\leq d$:
Take $R\leq \left\vert A_{1}\right\vert /2$. Suppose $r<R$ and that%
\begin{equation*}
\left\vert A_{k+1}\right\vert /2<R\leq \left\vert A_{k}\right\vert /2.
\end{equation*}%
Then $B(z,R)$ can intersect at most two (consecutive) sets $A_{i}$ for $%
i\leq k,$ (say $i=m,m+1),$ as well as possibly $\bigcup_{i=k+1}^{\infty
}A_{i}$. Assume%
\begin{equation*}
\left\vert A_{j}\right\vert \leq r<\left\vert A_{j-1}\right\vert
\end{equation*}%
where, of course, $j\geq k+1$. Since $\sum_{i=j}^{\infty }\left\vert
A_{i}\right\vert \leq 2\left\vert A_{j}\right\vert $, one ball of radius $r$
will cover $\bigcup_{i=j}^{\infty }A_{i}$. As $r\leq \left\vert
A_{k}\right\vert \leq \left\vert A_{m+1}\right\vert $, from (\ref{Adim1}) we
have 
\begin{eqnarray*}
N_{r}(B(z,R)\cap \mathcal{A}) &\leq &N_{r}(B(z,R)\cap (A_{m}\cup
A_{m+1}))+N_{r}\left( B(z,R)\cap \bigcup_{i=k+1}^{\infty }A_{i}\right) \\
&\leq &2C\left( \frac{R}{r}\right) ^{d}+\sum_{i=k+1}^{j-1}N_{r}(B(z,R)\cap
A_{i})+1
\end{eqnarray*}%
(where the sum is empty if $j-1<k+1$). Since $r<\left\vert A_{i}\right\vert $
for $i=k+1,\ldots,j-1$, from (\ref{Adim1}) we again see that%
\begin{eqnarray*}
N_{r}(B(z,R)\cap \mathcal{A}) &\leq &2C\left( \frac{R}{r}\right)
^{d}+C\sum_{i=k+1}^{j-1}\left( \frac{\left\vert A_{i}\right\vert }{r}\right)
^{d}+1 \\
&\leq &C^{\prime }\left( \frac{R}{r}\right) ^{d}+C^{\prime }\left( \frac{%
\left\vert A_{k+1}\right\vert }{r}\right) ^{d}\leq C^{\prime \prime }\left( 
\frac{R}{r}\right) ^{d}.
\end{eqnarray*}%
That proves that $\dim _{A}\mathcal{A}\leq d$ and hence $\dim _{qA}\mathcal{A%
}=d$.

\medskip

\textbf{Proof of Claim:}

Choose $\gamma \leq 1$ such that the diameter of $A_{k}\geq \gamma \alpha
_{m_{k}}\beta $ for all $k$. Temporarily fix $k$. Choose $R$ and $r<\min
(\left\vert A_{k}\right\vert ,R)$.

First, suppose there is some $j\in \mathbb{N}$ such that 
\begin{equation*}
\gamma \alpha _{m_{k}}\beta ^{j+1}/4<R\leq \gamma \alpha _{m_{k}}\beta ^{j}/4
\end{equation*}%
(in particular, $R<\left\vert A_{k}\right\vert )$. If $j>n_{k},$ then $2R$
is smaller than the smallest block in $A_{k}$ and thus $B(z,R)$ can
intersect at most two blocks in $A_{k}$. As there are at most $1/v$ gaps in
each block, 
\begin{equation*}
N_{r}(B(z,R)\cap A_{k})\leq \frac{2}{v}\leq C\left( \frac{R}{r}\right) ^{d}.
\end{equation*}

Hence assume $j\leq n_{k}$. Then $2R$ is less than the length of any block
of level $\leq j$ and thus $B(z,R)\cap A_{k}$ can intersect at most two
(consecutive) blocks of level $\leq j,$ as well as the interval $I$
in-between (where an in-between interval could mean the interval between the
left or right-most block of level $j$ and the endpoint of the set $A_{k})$.
The points in $\mathcal{A}$ from the two blocks of level at most $j$ can be
covered by $2/v$ balls of radius $r,$ hence 
\begin{equation*}
N_{r}(B(z,R)\cap A_{k})\leq \frac{2}{v}+N_{r}(B(z,R)\cap I).
\end{equation*}%
Notice that the interval $I$ will contain (at most) $2^{n-j}$ blocks of
level $n\geq j+1$. Also, observe that the interval between two consecutive
blocks of level $n$ (should it exist in $A_{k})$ has length at most%
\begin{equation*}
\sum_{i=n+1}^{\infty }2^{i-(n+1)}\alpha _{m_{k}}\beta ^{i}\leq \alpha
_{m_{k}}\frac{\beta ^{n+1}}{1-2\beta }.
\end{equation*}%
Thus if 
\begin{equation}
\alpha _{m_{k}}\frac{\beta ^{n+1}}{1-2\beta }<r\leq \alpha _{m_{k}}\frac{%
\beta ^{n}}{1-2\beta }\text{ for some }n\geq j+1,\text{ }  \label{rbd}
\end{equation}%
then each such subinterval can be covered by one ball of radius $r$. There
are at most $2^{n-j}$ such subintervals contained in $I$. Additionally, the
points in $\mathcal{A}$ from each of the blocks of levels $j+1,\ldots,n$
contained in $I$ can be covered by $1/v$ balls of radius $r$ and there are $%
\leq 2^{n-j}$ such blocks. So 
\begin{eqnarray*}
N_{r}(B(z,R)\cap I) &\leq &2^{n-j}+2^{n-j}/v\leq C^{\prime }2^{n-j} \\
&=&C^{\prime }(\beta ^{j-n})^{d}\leq C^{\prime \prime }(R/r)^{d}\text{.}
\end{eqnarray*}%
Thus for such $r$ we certainly have 
\begin{equation*}
N_{r}(B(z,R)\cap A_{k})\leq \frac{2}{v}+C^{\prime \prime }\left( \frac{R}{r}%
\right) ^{d}\leq C\left( \frac{\min (\left\vert A_{k}\right\vert ,R)}{r}%
\right) ^{d}
\end{equation*}%
for a suitable constant $C$ (recalling that $R<\left\vert A_{k}\right\vert $
and $R/r\geq 1$).

If (\ref{rbd}) does not hold, we must have 
\begin{equation*}
\alpha _{m_{k}}\frac{\beta ^{j+1}}{1-2\beta }<r\leq R\leq \gamma \alpha
_{m_{k}}\beta ^{j}/4.
\end{equation*}%
Then $B(z,R)$ is covered by a bounded number (independent of $j,k)$ of balls
of radius $r$ and that also suffices to prove 
\begin{equation*}
N_{r}(B(z,R)\cap A_{k})\leq C\left( \frac{R}{r}\right) ^{d}\leq C\left( 
\frac{\min (\left\vert A_{k}\right\vert ,R)}{r}\right) ^{d}
\end{equation*}%
for these $r.$

Otherwise, $R>\gamma \alpha _{m_{k}}\beta /4$. If (still) $R\leq \left\vert
A_{k}\right\vert ,$ then we argue similarly, taking as $I$ the full set $%
A_{k}$. Finally, suppose $R>\left\vert A_{k}\right\vert $. Then 
\begin{equation*}
N_{r}(B(z,R)\cap A_{k})\leq N_{r}(B(z^{\prime },\left\vert A_{k}\right\vert
)\cap A_{k})
\end{equation*}%
where $z^{\prime }\in A_{k}$. As $r<\left\vert A_{k}\right\vert ,$ the
previous work shows 
\begin{equation*}
N_{r}(B(z^{\prime },\left\vert A_{k}\right\vert )\cap A_{k})\leq C\left( 
\frac{\left\vert A_{k}\right\vert }{r}\right) ^{d}\leq C\left( \frac{\min
(\left\vert A_{k}\right\vert ,R)}{r}\right) ^{d}.
\end{equation*}%
This completes the proof of the claim.

\medskip

\textbf{Conclusion of the proof for general case of }$\Phi \rightarrow 0$%
\textbf{:}

Lastly, we remark that the same arguments show that if $\Phi \rightarrow 0,$
then for each $d\in \lbrack \overline{\dim }_{\Phi }C_{a},1)$ there is some $%
E=\mathcal{A}\cup \mathcal{B}\in \mathcal{C}_{a}$ with $\,\dim _{qA}\mathcal{%
A}=\dim _{A}\mathcal{A}$, so that also $\overline{\dim }_{\Phi }\mathcal{A}%
=d $. Further, $\overline{\dim }_{\Phi }\mathcal{B}=\overline{\dim }_{\Phi
}C_{a}$ and thus $\overline{\dim }_{\Phi }E=d$ by the union result,
Proposition \ref{union}. Since we have $1=\dim _{qA}D_{a}=\overline{\dim }%
_{\Phi }D_{a}$ the proof is complete when $\Phi \rightarrow 0$.
\end{proof}


\begin{thebibliography}{9}
\bibitem{A1} P. Assouad, U.E..R. Math\'{e}matique, Universit\'{e} Paris XI,
Orsay. Th{\`{e}}se de doctorat d'\'{E}tat, \emph{Publications Math{\'{e}}%
matiques d'Orsay,} No. 223-7769, 1977. 

\bibitem{A2} P. Assouad, \'{E}tude d'une dimension m\'{e}trique li\'{e}e 
\`{a} la possibilit\'{e} de plongements dans $\mathbb{R}^{n}$. \emph{%
C.R.Acad. Sci. Paris S\'{e}r. A-B,} \textbf{288} (1979), no. 15, A731--A734. 

\bibitem{BT} A.S. Besicovitch and S.J. Taylor, On the complementary
intervals of a linear closed set of zero {L}ebesgue measure. \emph{J. London
Math. Soc.,} \textbf{29} (1954), 449--459. 

\bibitem{CDW} H. Chen, Y. Du and C. Wei, Quasi-lower dimension and
quasi-Lipschitz mapping. \emph{Fractals,} \textbf{25}(3), 1-9, 2017. 

\bibitem{CWC} H. Chen, M. Wu and Y. Chang, Lower Assouad type dimensions of
uniformly perfect sets in doubling metric space. \emph{Fractals}, \textbf{28} (2020)
doi:10.1142/S0218348X20500395.

\bibitem{Fal} K. Falconer, {\em Techniques in fractal geometry}. John Wiley
\& Sons Ltd., Chichester, 1997. 

\bibitem{FFK} K. Falconer, J. M. Fraser and T. Kempton, Intermediate
dimensions. \emph{Math. Z., to appear}. doi:10.1007/s00209-019-02452-0. Preprint, 2018. arxiv:1811.06493 [math.MG]

\bibitem{FTrans} J.M. Fraser, Assouad type dimensions and homogeneity of
fractals. \emph{Trans. Amer. Math. Soc.,} \textbf{366} (2014), no. 12, 6687--6733. 

\bibitem{FrSurvey} J.M.\ Fraser, Interpolating between dimensions. \emph{Proceedings of Fractal Geometry and Stochastics VI, to appear.} Preprint, 2019. arxiv:1905.11274 [math.MG] 

\bibitem{FTale} J.M. Fraser, K.G. Hare, K.E. Hare, S. Troscheit and H. Yu,
The Assouad spectrum and the quasi-Assouad dimension: a tale of two spectra. 
\emph{Ann. Acad. Sci. Fenn. Math}., \textbf{44} (2019), no. 1, 379-387. 

\bibitem{FHOR} J.M. Fraser, A.M. Henderson, E.J. Olson, and J.C. Robinson,
On the {A}ssouad dimension of self-similar sets with overlaps. \emph{Adv.
Math., }\textbf{273} (2015), 188--214. 

\bibitem{FYAdv} J.M. Fraser and H. Yu, New dimension spectra: finer
information on scaling and homogeneity. \emph{Adv. Math., }\textbf{329} (2018),
273--328. 

\bibitem{FYInd} J.M. Fraser and H. Yu, Assouad type spectra for some fractal
families. \emph{Indiana Univ.~Math.~J.} \textbf{67} (2018), no. 5, 2005--2043.  

\bibitem{GH} I. Garc\'{\i}a and K.E. Hare, Properties of Quasi-Assouad
dimension. \emph{Ann. Acad. Sci. Fenn. Math., to appear}. Preprint, 2017. arxiv:1703.02526 [math.CA]

\bibitem{GHM} I. Garc\'{\i}a, K.E. Hare and F. Mendivil, Assouad dimensions
of complementary sets. \emph{Proc. Roy. Soc. Edinburgh Sect. A} \textbf{148} (2018), 
517-540. 

\bibitem{GHMArXiv} I. Garc\'{\i}a, K.E. Hare and F. Mendivil, Almost sure
Assouad-like dimensions of complementary sets. Preprint, 2019. arxiv:1903.07800 [math.CA]

\bibitem{HMZ} K.E. Hare, F. Mendivil, and L. Zuberman, The sizes of
rearrangements of Cantor sets. \emph{Can. Math. Bull.,} \textbf{56} (2013), no. 2, 354--365. 

\bibitem{HT18} K.E. Hare and S. Troscheit, \emph{Math. Proc. Camb. Phil. Soc., to appear}. doi:%
{10.1017/S0305004119000458}. Preprint, 2018. arxiv:1812.05573 [math.MG]

\bibitem{Ha} J. Hawkes, Random re-orderings of intervals complementary to a
linear set. \emph{Quart. J. Math. Oxford Ser. }\textbf{35} (1984), 165-172. 

\bibitem{Heinonen} J. Heinonen, Lectures on analysis on metric spaces. \emph{%
Universitext. Springer-Verlag, New York,} 2001. 

\bibitem{KLV} A. K{\"{a}}enm{\"{a}}ki, J. Lehrb{\"{a}}ck and M. Vuorinen, 
Dimensions, Whitney covers, and tubular neighborhoods. \emph{Indiana
Univ. Math. J.} \textbf{62} (2013), no. 6, 1861--1889. 

\bibitem{KR} A. K\"{a}enm\"{a}ki and E. Rossi, Weak separation condition,
Assouad dimension, and Furstenberg homogeneity. \emph{Ann. Acad. Sci. Fenn.
Math}. \textbf{41} (2016), 465-490.

\bibitem{L1} D.G. Larman, \newblock A new theory of dimension. \newblock%
\emph{Proc. Lond. Math. Soc.} (3) \textbf{17} (1967), 178--192. \MR{203691} 

\bibitem{L2} D.G. Larman, On Hausdorff measure in finite-dimensional compact
metric spaces. \emph{Proc. Lond. Math. Soc.} (3) \textbf{17} (1967), 193-206. 

\bibitem{Li} W.W. Li, W.X. Li, J.J. Miao, and L.F. Xi, Assouad dimensions of
Moran sets and Cantor-like sets. \emph{Front. Math. China} \textbf{11} (2016) 
705-722. 

\bibitem{LX} F. L\"{u} and L. Xi, Quasi-Assouad dimension of fractals. \emph{%
J. Fractal Geom.} \textbf{3} (2016), no. 2, 187-215. 

\bibitem{Luu} J. Luukkainen, Assouad dimension: antifractal metrization,
porous sets, and homogeneous measures. \emph{J. Korean Math Soc.} \textbf{35%
} (1998), no. 1, 23--76. 

\bibitem{MT} J. Mackay and J. Tyson, Conformal dimension. \emph{Univ.
Lecture Series} \textbf{54}, 2010. 

\bibitem{Tricot} C. Tricot, Curves and fractal dimension. \emph{%
Springer-Verlag, New York}, 1995. With a foreword by Michel Mend\`{e}s
France, Translated from the 1993 French original. 

\bibitem{Tr} S. Troscheit, Assouad spectrum thresholds for some random
constructions. \emph{Canad. Math. Bull.} \textbf{63} (2020), no. 2, 434--453.. 
\end{thebibliography}
\end{document}